\documentclass{article}

\usepackage[utf8]{inputenc} 
\usepackage[T1]{fontenc}    
\usepackage{hyperref}       
\usepackage{url}            
\usepackage{booktabs}       
\usepackage{amsfonts}       
\usepackage{nicefrac}       
\usepackage{microtype}      
\usepackage{amsmath}
\usepackage{amsthm}
\usepackage[group-separator={,}]{siunitx}
\usepackage{thmtools}
\usepackage{thm-restate}
\usepackage{subcaption}

\usepackage{mathtools}

\usepackage{float}
\usepackage{xcolor}
\usepackage[inline]{enumitem}
\usepackage{xr}
\usepackage{cuted}
\usepackage{amssymb}
\usepackage[numbers]{natbib}
\usepackage{authblk}
\usepackage{enumitem}

\usepackage[margin=1in]{geometry}

\usepackage{algorithm}
\usepackage{algorithmic}

\usepackage{todonotes}

\newtheorem{theorem}{Theorem}[section]

\newtheorem{lemma}[theorem]{Lemma}
\newtheorem{corollary}[theorem]{Corollary}
\newtheorem{assumption}[theorem]{Assumption}

\newtheorem{example}{Example}[section]
\newtheorem{remark}{Remark}

\newcommand{\mX}{\mathcal{X}}
\newcommand{\mY}{\mathcal{Y}}
\newcommand{\mP}{\mathcal{P}}
\newcommand{\mS}{\mathcal{S}}
\newcommand{\R}{\mathbb{R}}
\newcommand{\N}{\mathbb{N}}
\newcommand{\E}{\mathbb{E}}
\newcommand{\tS}{\tilde{\mS}}
\newcommand{\tP}{\tilde{\mP}}

\newcommand{\scen}[1]{s_{#1}}
\newcommand{\newscen}[1]{\tilde{s}_{#1}}
\newcommand{\prob}{\mathbb{P}}
\newcommand{\tx}{\tilde{x}}

\DeclareMathOperator*{\argmax}{arg\,max}

\title{Scenario Reduction for Distributionally Robust Optimization}

\author[1]{Kevin-Martin Aigner}
\author[1]{Sebastian Denzler\footnote{Corresponding author. E-mail address: \href{mailto:sebastian.denzler@fau.de}{sebastian.denzler@fau.de}}}
\author[1]{Frauke Liers}
\author[2]{Sebastian Pokutta}
\author[3]{Kartikey Sharma}

\affil[1]{Friedrich-Alexander-Universität Erlangen-Nürnberg, Germany}
\affil[2]{Zuse Institute Berlin, Germany}
\affil[3]{Indian Institute of Technology Delhi, India}

\begin{document}
\maketitle

\begin{abstract}
Stochastic and (distributionally) robust optimization problems often become computationally challenging as the number of scenarios or data points increases.
Scenario reduction is therefore a key technique for improving tractability.
We introduce a general scenario reduction method for distributionally robust optimization (DRO), which includes stochastic and robust optimization as special cases.
Our approach constructs the reduced DRO problem by projecting the original ambiguity set onto a reduced set of scenarios.
Under mild conditions, we establish bounds on the relative quality of the reduction.  
The methodology is applicable to random variables following either discrete or continuous probability distributions, with representative scenarios appropriately selected in both cases.
Given the relevance of optimization problems with linear and quadratic objectives, we further refine our approach for these settings.  
Finally, we demonstrate its effectiveness through numerical experiments on mixed-integer benchmark instances from MIPLIB and portfolio optimization problems.  
Our results show that the proposed approximation significantly reduces solution time while maintaining high solution quality with only minor errors.
\end{abstract}

\section{Introduction}
\label{sec:introduction}


Uncertainties can significantly impact both the feasibility and optimality of solutions to optimization problems.
To ensure that solutions remain effective under real-world conditions, it is essential to incorporate strategies that account for parameter uncertainties.
Two primary approaches for addressing uncertainty in optimization problems are \emph{Stochastic Optimization (SO)} and \emph{Robust Optimization (RO)}.
In SO, uncertain parameters are modeled as random variables, and solution methods typically rely on knowledge of their probability distributions.
Traditionally, SO seeks solutions that are optimal in terms of expected outcomes or, more broadly, in relation to chance constraints or various risk measures, see~\cite{BirgeLouveaux1997}.
Robust Optimization (RO) is generally applied when the probability distribution of uncertain parameters is unknown or when a stronger guarantee of feasibility is required~\cite{Ben-TalElGhaouiNemirovski2009, robust_bertsimas_hertog}.
It aims to find solutions that perform optimally against worst-case realizations of uncertain parameters selected from a predefined uncertainty set.

In practical SO problems, information about the underlying probability distribution is often limited.
Recent research has focused on developing approaches that combine the advantages of both SO and RO to efficiently achieve high-quality protection against uncertainty. 
Specifically, \emph{Distributionally Robust Optimization (DRO)} seeks to solve a ``robust'' version of a stochastic optimization problem when there is incomplete knowledge of the probability distribution~\cite{DRO_book}.
The advantage of DRO lies in its ability to protect solutions against relevant uncertainties while avoiding the overly conservative incorporation of irrelevant uncertainties sometimes observed in classic robust approaches.
The articles~\cite{drosurvey1, drosurvey2} provide a comprehensive overview of DRO literature and we refer the interested reader to those for an in-depth exposition.

The size and complexity of optimization problems under uncertainty
increase with the number of scenarios or data points, resulting in long solution times.
To manage computational complexity, scenario reduction techniques are often applied.
These methods aim to reduce the number of scenarios while retaining the essential information about the uncertainty, thereby making the optimization problem more tractable without significantly compromising solution quality.
As a well-established approach, scenario reduction is applied in stochastic and robust optimization~\citep{goerigk_scen_reduction, roemisch-scenredSO}.
In this paper, we address the challenge of reducing the scenario set
size in DRO for discrete and continuous uncertainties with a
monotonically homogeneous uncertain objective function. 
To achieve this, we select a set of representative scenarios and formulate an approximate DRO problem.
This approach generates solutions that remain well-protected against uncertainty while still delivering reliable results for the original problem.
The main idea is to use clustering methods in order to aggregate similar scenarios.
Specifically, we perform scenario clustering in two ways: first, by solving a mixed-integer optimization problem that minimizes the clustering approximation factor, and second, using the well-known $k$-means algorithm.
We refine and extend our approach to address problems with quadratic objectives, where the clustering task is formulated as a mixed-integer semidefinite program.
To construct the approximated DRO problem, we aggregate the probabilities in the ambiguity set according to their representative scenarios.
Computational experiments on benchmark instances from the MIPLIB and portfolio optimization problems demonstrate that our scenario reduction typically induces small errors only.
In instances of non-linear scenario dependence, our method demonstrates superior accuracy in approximating the objective function.
Additionally, the effectiveness of $k$-means as a fast clustering method is shown by achieving similar approximation guarantees for the representative scenarios 
with significantly faster computation.

\paragraph*{Problem setting}

We consider the problem of minimizing an uncertain objective function
\(f \colon (x,s) \in \R^n \times \R^m \mapsto f(x,s)\in \R\) with
\(n,m\in \N\) over a feasible set \(\mX\subseteq\R^n\) that is allowed
to contain both discrete as well as continuous variables.
The objective depends on the decision variable \(x\in \mX\) 
and a random vector \(s\in\mS\) with realizations inside the scenario set \(\mS\subseteq\R^m\).

Given a probability distribution \(\prob\) of the uncertain parameter
\(s\) with support \(\mS\), the stochastic program reads
\begin{equation}
	\inf_{x\in\mX} \E_{s\sim \prob} \left[f(x,s)\right],  \label{Eq:SO}\tag{SO}
\end{equation}
where \(\E_{s\sim \prob}\left[f(x,s)\right]=\int_{s\in\mS} f(x,s) \prob(\text{d}s)\) denotes the expected value for the given probability distribution~\(\prob\) over~\(s\).
Solving the stochastic optimization problem~\eqref{Eq:SO} requires knowledge about the underlying probability distribution of the random vector \(s\).
In many settings, such information may not be available or may be incomplete. 
A robust protection of~\eqref{Eq:SO} against such incomplete knowledge about the probability distribution can be achieved by the distributionally robust optimization problem
\begin{equation}
	\inf_{x\in\mX} \sup_{\prob \in \mP}  \E_{s\sim \prob} \left[f(x,s)\right],  \label{Eq:DRO}	\tag{DRO}
\end{equation}
where \(\mP\) is an ambiguity set of probability distributions.
Our goal with solving the DRO problem is to minimize the worst-case expected cost over all distributions in the ambiguity set \(\mP\). 
The key challenge in solving problem~\eqref{Eq:DRO} lies in developing
algorithmically tractable robust counterparts. Their computational
complexity typically depends not only on the difficulty of optimizing for $x$ but also 
strongly depends on the difficulty of optimizing over \(\mP\). 
In the following, we will differentiate between continuous and
discrete probability distributions.

It is worth mentioning that distributional robustness includes both
stochastic as well as classical robust optimization as a special case.
Indeed, if the ambiguity set \(\mP\) equals the set of all possible probability distributions of \(s\) over \(\mS\), the DRO problem~\eqref{Eq:DRO} is equivalent to the strict robust optimization problem
\begin{equation}
	\min_{x\in\mX} \max_{s\in\mS} f(x,s).
\end{equation}
For discrete random variables, the maximum size ambiguity set is the probability simplex $\mP=\{ p \in [0,1]^{|\mS|}: \sum_{k=1}^{|\mS|} p_k = 1\}$.
If the ambiguity set consists only of one single probability distribution \(\mP=\{\prob^*\}\), the DRO problem~\eqref{Eq:DRO} is equal to the stochastic problem~\eqref{Eq:SO} with respect to the distribution \(\prob^*\).


There are various ways of constructing an ambiguity set to robustify
stochastic programs like~\eqref{Eq:SO} in the literature to get an
algorithmically tractable reformulation of the distributionally robust
counterpart~\eqref{Eq:DRO}. 

\paragraph*{Motivation for scenario reduction}

An example of a data-driven ambiguity set \(\mP\) for a discrete random variable with finite support \(s\in \mS\coloneqq \{\scen{1},\ldots,\scen{|\mS|}\}\subseteq \R^{n}\), for \(n\in\N\), using box ambiguity sets with confidence intervals is of the form
	$\mP \coloneqq \left\{ p\in \left[ 0,1\right]^{|\mS|} : l\leq p \leq u \text{, } \sum_{k=1}^{|\mS|} p_k=1 \right\}$.
Here, \(l, u \in [0,1]^{|\mS|}\) are the lower and upper bounds of the confidence intervals.
This ambiguity set is also considered in~\cite{DRO_over_time}.

In this case, the distributionally robust counterpart~\eqref{Eq:DRO} can be reformulated, as the inner maximization problem forms a linear program. Then, by strong duality,~\eqref{Eq:DRO} is equivalent to
\begin{subequations}\label{Eq:DRO_reformulated_box}
	\begin{align}
		\min_{x,z,\lambda,\mu} \quad &z - l^\top \lambda + u^\top \mu
		\label{Eq:DRO_reformulated_box:a} \\
		\text{s.t. } \quad  & z- \lambda_k + \mu_k \geq f(x,\scen{k}) \quad \forall k=1,\ldots,|\mS|,
		\label{Eq:DRO_reformulated_box:b} \\
		&\lambda,\mu\geq 0		\label{Eq:DRO_reformulated_box:c} \\
		&x\in\mX,\text{ }z\in\R,\text{ }\lambda,\mu \in\R^{|\mS|}.
		\label{Eq:DRO_reformulated_box:d}
	\end{align}
\end{subequations}
Here, the dual variables \(\lambda_k\) and \(\mu_k\) price the ambiguity for scenario \(\scen{k}\in\mS\). 
This problem is of the same problem class as~\eqref{Eq:SO}; however, it is larger in size.
The reformulated problem grows linearly with the number of scenarios
that may become prohibitive if the cardinality of S is large.
Thus, the difficulty of solving~\eqref{Eq:DRO_reformulated_box} depends on the complexity of \(f\) and the cardinality of \(\mS\).
Reducing the scenario set leads to smaller problems that remain
algorithmically tractable.
\noeqref{Eq:DRO_reformulated_box:a,Eq:DRO_reformulated_box:b,Eq:DRO_reformulated_box:c,Eq:DRO_reformulated_box:d}

\subsection{Related Work}

%


Scenario reduction is a key technique used in reducing the complexity and the size of stochastic programming problems~\cite{roemisch-scenredSO}.
A popular approach is to generate new scenarios and assigned probabilities in order to minimize their distance to the original probability like in~\cite{dupavcova2003scenario, henrion2009scenario}.
This approach focuses on the uncertainty in the underlying probability distribution.
In multistage settings, stochastic processes have been approximated by scenario trees~\cite{hochreiter2007financial, pflug2001scenario}.
More recently, a lot of attention has also been paid to problem-based scenario reduction methods where the optimization objective itself is leveraged in the construction of the metrics used to simplify scenarios~\cite{bertsimas2023optimization,keutchayan2023problem, problem_driven_scen_reduction, zhang2023optimized}. 
For this purpose, a variety of techniques have been developed to aggregate scenarios and simplify scenario trees, such as clustering~\cite{kammammettu2023scenario}, moment matching~\cite{carino1998concepts}, objective approximation~\cite{zenios1993constructing}, etc.
Machine learning techniques that attempt to learn scenario representations have also become popular~\cite{wu2022learning}.
Some papers have primarily focused on ineffective scenarios, which are scenarios whose removal does not affect the risk-averse objective function~\cite{arpon2018scenario}, or have aimed to incorporate the risk measure into the distance metric used for aggregation~\cite{pineda2010scenario}.
However, these methods are not directly applicable to the DRO case, as they typically assume a given probability distribution over the original scenarios.
In the DRO modeling approach, there is an entire set of possible
probability distributions that has to be taken into account.
It is very relevant whenever computationally efficient approaches
shall be implemented in challenging applications under uncertainty,
for example in electricity networks~\cite{aigner_curtailment, aigner_dvine}.

Our work shares similarities with scenario reduction methods as they were developed for
robust optimization in~\cite{goerigk2019representative, goerigk_scen_reduction}.
The authors propose models to reduce the number of scenarios for one- and two-stage robust optimization problems with discrete uncertainty in the objective function.
The reduction is done in a way such that an approximation quality holds.
The recent work~\cite{goerigk_new} improves on that by also taking the feasible set into consideration when reducing the scenarios.

There has been limited research addressing reduction methods for DRO.
Similarly to~\cite{arpon2018scenario}, ineffective scenarios can be identified for ambiguity sets~\cite{rahimian2022effective, rahimian2019identifying}.
Clustering-based methods that aggregate scenarios and then create separate ambiguity sets for each of them have also been proposed~\cite{zhang2023stochastic}.
Motivated by scenario reduction for robust~\cite{goerigk_scen_reduction} and stochastic optimization~\cite{bertsimas2023optimization}, we develop a probability-driven scenario reduction approach for distributionally robust optimization problems compatible with both discrete and continuous probability distributions with a provable approximation quality. 

Dimension reduction methods typically focus on reducing the dimensionality of scenario vectors themselves.
Common approaches for this goal are feature selection~\cite{feature_selection}, principal component analysis~\cite{PCA_Pearson, pca}, random projections~\cite{random_projection} or compressed sensing~\cite{compressed_sensing}.
These approaches aim to simplify the representation of scenarios by reducing the number of features, thereby making the data more manageable without decreasing the number of scenarios.
In contrast, our focus is on reducing the number of scenarios and computing suitable representative scenarios while ensuring provable approximation guarantees. 
Importantly, our framework does not rely on structural assumptions on the ambiguity set and applies to both discrete and continuous supports. 
This generality, however, comes at the cost of potentially more conservative bounds compared to reduction methods that explicitly exploit problem-specific ambiguity set geometry or objective structure.


\subsection{Contribution}

The main contributions of this paper are as follows:
\begin{enumerate}
	\item We present a \emph{scenario reduction approach for distributionally robust optimization} with provable worst-case approximation quality for problems with an objective function that is monotonically homogeneous in the uncertainty.
	\item Unlike existing methods, we impose no structural requirements on the ambiguity set and conceptually handle both discrete and continuous supports.
	\item We provide an explicit mixed-integer optimization formulation for an \emph{optimal clustering} along with a fast alternative based on the well-known \emph{$k$-means}
          algorithm~\cite{kmeans_macqueen}. For quadratic objectives, we use a mixed-integer semidefinite program for this purpose. 
	\item We illustrate the performance of our methods through numerical experiments on benchmark mixed-integer linear problems from the MIPLIB library and practically relevant applications from convex quadratic portfolio optimization on real-world data.
\end{enumerate}


\subsection{Outline}
The introduction of our work is followed by setting up the framework and providing the key approximation results in Section~\ref{sec:scenario_reduction} for DRO problems with monotone and homogeneous uncertain objective functions together with an arbitrary ambiguity set for discrete or continuous random variables.
Section~\ref{sec:scenario_partitions} discusses the various approaches that can be used to reduce and simplify the set of scenarios by partitioning the scenario set via k-means clustering or integer programming.
Subsequently, Section~\ref{sec:dimen_red_amb_sets} discusses the projection of the ambiguity set over the original set of scenarios to the lower dimensional DRO problem. 
Section~\ref{sec:extensions} extends our work to the case of uncertain quadratic objectives, where the uncertain parameter is a symmetric quadratic positive semidefinite matrix.
Finally, we illustrate the performance of our methods through
numerical experiments in Section~\ref{sec:experiments} and summarize
our findings in the conclusion.

\section{Scenario Reduction for DRO with Approximation Guarantee}
\label{sec:scenario_reduction}

To successfully solve DRO problems when many or a continuous set of scenarios are present, the scenarios need to be reduced.
In addition, the problem~\eqref{Eq:DRO} needs to be simplified by reducing the problem size. 
Our framework achieves this by partitioning the original scenario set into clusters or regions.
We then calculate a representative scenario for each new set in the partition. 
By focusing on uncertain monotone and positive homogeneous objective functions, we cover a wide range of optimization problems, allowing us to obtain approximation guarantees for the objective values of the representative scenarios. 

Our goal for this framework is to obtain worst-case guarantees on the approximation quality of the solution \(\tilde{x}\in\mX\) obtained from solving the optimization problem with the reduced set of scenarios. 
Such an approximation guarantee typically looks like
\begin{equation}
\sup_{\prob\in\mP} \E_{s\sim \prob} \left[f(\tilde{x},s)\right] \leq \alpha \inf_{x\in\mX}  \sup_{\prob\in\mP} \E_{s\sim \prob} \left[f(x,s)\right],\label{Eq:Approx_quality}
\end{equation}
with \(\alpha \geq 1\).
This means that the approximate solution has a worst-case objective guarantee that is no worse than a factor \(\alpha\) when compared to the original worst-case objective value.
In this setting, we call \(\tilde{x}\) an \emph{\(\alpha\)-approximation} of problem~\eqref{Eq:DRO}.
Our results holds for distributions with both discrete and continuous scenario sets unless specified otherwise.

\subsection{Aggregating Scenarios}\label{Sec:AggregatingScen}
Given a set of scenarios, we aim to choose a representative scenario such that all scenarios in the set can be bounded by scaling of the representative scenario.
In addition, it is required that the representative scenario can be bounded by scaling any of the original scenarios.
We thus focus our attention on functions whose response to scaled arguments is bounded by a scaled original value.
Specifically, we focus our attention on functions which satisfy $f(x,s) \leq \alpha f(x, \tilde{s}) \text{ for all } s, \tilde{s} \in \R^m$,
such that $s \leq \alpha \tilde{s}$ holds component-wise with $\alpha>0$.
In order to obtain an approximation guarantee for the approximate solution, we assume monotonicity and homogeneity for the objective function \(f(x,s)\) and all feasible $x\in\mX$.
We give a formal definition of this assumption as follows.


\begin{assumption}
	\label{assump:func_assumps}
	The objective function \(f\colon \mX \times \mS \rightarrow \R, (x,s) \mapsto f(x,s)\) fulfills the following conditions.
	\begin{enumerate}[label= (\roman*)]
		\item Monotonicity: For an arbitrarily fixed \(x\in\mX\), the following implication holds for all \(s,\tilde{s}\in\R^m:\) $s\leq\tilde{s} \Rightarrow f(x,s)\leq f(x,\tilde{s})$.
		\item Positive homogeneity: for \(x\in\mX, s\in\R^m\) and \(\alpha > 0\), we have \(f(x, \alpha s)\leq C(\alpha) f(x,s)\) for some function $C:\R_+ \rightarrow \R_+$. 
	\end{enumerate}
\end{assumption}

This assumption ensures that we can bound the change in the function when the original scenarios are replaced by a representative.
This includes positive homogeneity of any arbitrary degree i.e., \(f(x,\alpha s)\leq \max(\alpha, \alpha^k) f(x,s)\) for fixed $x\in\mX$ without any change in the analysis we perform for $C(\alpha)=\alpha$ for ease of notation.
We note that a wide class of problems satisfies the assumptions above.
Next, we give examples of functions satisfying Assumption \ref{assump:func_assumps}.

\begin{example}
	Let $g\colon \mX \rightarrow \R^{m}, x \mapsto g(x)$. The following classes of functions fulfill Assumption \ref{assump:func_assumps}:
	\begin{enumerate}[label= (\roman*)]
		\item Linear functions: \(f(x,s)=g(x)^\top s\) with $g(x)\geq 0$ componentwise. This also includes separable objective functions $f(x,s)=g(x)^\top h(s)$ by redefining the scenarios as $s' = h(s)$ with $h\colon \mS \rightarrow \R_+^{m}, s \mapsto h(s)$.
		\item Functions that scale with the norm of the uncertainty: \(f(x,s) = g(x)\|s\|\).
		\item Nondecreasing positive homogeneous functions of an arbitrary degree \(k\geq 1\) such as \(f(x,s) = (x^{\top}  s)^k\) if \(x \geq 0\).
		\item Optimal value functions of the form $f(x,s)=c^\top x + \min_{y\in \mY(x)} g(y)^\top s$ with $g(y)\geq 0$ componentwise and $\mY(x)$ as the feasible set of the second-stage problem.
	\end{enumerate}


	
\end{example}

Assumption~\ref{assump:func_assumps} is fulfilled by a wide range of practically relevant optimization problems. Linear combinatorial optimization problems, such as shortest path, assignment, scheduling, or network design, often have objectives of the form \(f(x, s) = g(x)^\top s\) with \(g(x) \geq 0\) and cost vectors \(s \in \R^m_+\), thus satisfying monotonicity and positive homogeneity; see, e.g.,~\cite{network_flow, wolsey}. Convex quadratic programs, such as portfolio optimization~\cite{portfolio}, also fulfill the assumption: here, \(f(x, s) = x^\top \Sigma x\), where \(\Sigma\) denotes a scenario representing a positive semidefinite covariance matrix. The function is then positively homogeneous of degree two and monotonic if \(x \geq 0\). Furthermore, two-stage stochastic optimization problems with linear second-stage costs~\cite{birge2011introduction}, given by \(f(x,s) = c^\top x + \min_{y \in \mY(x)} g(y)^\top s\)\ with \(g(y) \geq 0\), also satisfy the assumption. Such models frequently arise in applications like supply chain management~\cite{nagar2008supply} or energy systems~\cite{AIGNER2022318}.


In order to discuss the properties of a representative scenario, we
assume that \(\{\mS_1,\ldots \mS_K\}\subset\mS\) is a partition
of the scenario set \(\mS\) of size \(K\ll |\mS|\) such that
\(\mS=\cup_{j=1}^K \mS_j\) and \(\mS_{j}\cap \mS_{k} =\emptyset\) for
all \(j\neq k\).
We discuss how to obtain suitable partitions in Section \ref{sec:scenario_partitions}.
Given a partition, we proceed by constructing a new scenario set \(\tS =\{\newscen{1},\ldots, \newscen{K}\}\) where scenario \(\newscen{j}\) is the representative scenario for cluster \(\mS_{j}\) for all \(j=1,\ldots, K\).

Next, we make a assumption regarding the original and representative scenarios. 

\begin{assumption}
	\label{assump:pos_scenario_sets}
	The original scenario set \(\mS\) and the set of representative scenarios \(\tS\) are both strictly positive i.e., \(s > 0 \text{ for all } s \in \mS\) and \(\newscen{} > 0 \text{ for all } \newscen{} \in \tS\).
\end{assumption}
This assumption allows us to leverage the positive homogeneity property of the objective function to bound its value at the representative scenarios by simply scaling the objective function at the original scenario and vice-versa, as explained by the following lemma.
Note that the positivity assumption is often naturally satisfied in combinatorial or related practically relevant optimization problems.

\begin{lemma}
	\label{assump:scenario_bounds}
	Let $\mS$ be compact. If the original scenario set \(\mS\) and the representative scenario set \(\tS\) satisfy Assumption~\ref{assump:pos_scenario_sets} and if \(\mS\) is compact, then given a partition \(\{\mS_1,\dots,\mS_k\}\) of \(\mS\) we have:
	\begin{enumerate}[label=(\roman*)]
		\item There exists an \(\alpha>0\), such that for every \(j=1,\ldots,K\) it holds \(s \leq \alpha \newscen{j}\) for all \(s\in \mS_{j}\).
		\item There exists a \(\beta>0\), such that for every \(j=1,\ldots,K\) it holds \(\newscen{j}\leq \beta s\) for all \(s\in \mS_{j}\).
	\end{enumerate}
\end{lemma}
\begin{proof}
Since the set \(\mS\subset\R^m_+\) is compact, there exist points \(\overline{s},\underline{s}\in\R^m_+\) such that \(s \leq \overline{s} \) and \(s \geq \underline{s} \) hold componentwise for all \( s \in \mS\).
Similarly, for all partitioned sets \(\mS_j\) there exist points \(\overline{s}_j,\underline{s}_j \in\R^m_+\) such that \(s \leq \overline{s}_j \) and \(s \geq \underline{s}_j \) for all \( s \in \mS_j\).
Denote the components of \(\underline{s}_j\) and \(\overline{s}_j\) with index $i$ by \(\underline{s}_{ji}\) and \(\overline{s}_{ji}\), respectively. 
Let \(\tilde{s}_j\) be the representative scenario for set \(\mS_j\).
Then we have 
\[
\alpha = \max_{j=1,\dots,K} \alpha_j \text{ with } \alpha_j = \max_{i=1,\dots,m}\frac{\overline{s}_{ji}}{\tilde{s}_{ji}},
\]
and similarly
\[
\beta = \max_{j=1,\dots,K} \beta_j, \text{ where } \beta_j = \max_{i=1,\dots,m}\frac{\tilde{s}_{ji}}{\underline{s}_{ji}}.
\]
Given these \(\alpha\) and \(\beta\), we obtain the desired bounds componentwise for $i=1,\ldots,m$:
\[
\alpha \tilde{s}_{ji} = \left(\max_{j'=1,\dots,K} \alpha_{j'}\right) \tilde{s}_{ji} \geq \alpha_{j} \tilde{s}_{ji} = \left(\max_{i'=1,\dots,m}\frac{\overline{s}_{ji'}}{\tilde{s}_{ji'}}\right) \tilde{s}_{ji} \geq \frac{\overline{s}_{ji}}{\tilde{s}_{ji}} \tilde{s}_{ji} = \overline{s}_{ji} \geq s_{i} \;\;\forall s \in \mS_{j}.
\]
Similarly for any scenario \(s \in \mS_{j}\), we compute componentwise for $i=1,\ldots,m$:
\[
\beta s_{i} = \left(\max_{j'=1,\dots,K} \beta_{j'}\right) s_{i} \geq \beta_{j} s_{i} = \left(\max_{i'=1,\dots,m}\frac{\tilde{s}_{ji'}}{\underline{s}_{ji'}}\right) s_{i} \geq \frac{\tilde{s}_{ji}}{\underline{s}_{ji}} s_{ji} \geq \frac{\tilde{s}_{ji}}{\underline{s}_{ji}} \underline{s}_{ji} = \tilde{s}_{ji},
\]
where the last inequality holds because \(\underline{s}_{j} \leq s \;\forall\; s \in \mathcal{S}_{j}\) componentwise. 
This concludes the proof.
\end{proof}

So far, we have bounded the original and the representative scenarios
with respect to each other. 
We consider next the impact of replacing the original scenarios with the representative scenarios on the ambiguity set. The latter specifies the probabilities of these scenarios. 


\subsection{Reduced Ambiguity Set}
Reducing the size of the scenario set also changes the probability distributions over the scenarios and, consequently, the ambiguity set.
There are multiple ways by which such a transformation can take place. 
For example, we may seek the new probability distribution over the reduced set that minimizes the distance between itself and the original distribution. 
However, since these new scenarios are supposed to be representative of the original scenarios, we opt for a simpler alternative by aggregating the probabilities in a straightforward manner. 
Given a probability distribution \(\prob\in \mP\) over the original scenarios \(\mS\), we modify it to a probability distribution over the reduced scenario set \(\tS\)  with $|\tS|=K$.
We do this by computing the probability vector \(\tilde{p}\in [0,1]^K\) as 
\begin{equation*}
	\tilde{p}_j = \int_{s\in \mS_j} \prob(\text{d}s),
\end{equation*}
where \(\tilde{p}_j\) is the probability associated with
\(\newscen{j}\) which is a representative scenario of set
\(\mS_j\). 
By aggregating every distribution in the set, we thus adapt an ambiguity set over the original set \(\mP\) to an ambiguity set over the reduced set denoted by
\begin{equation}
	\tilde{\mP} \coloneqq \left\{ \tilde{p}\in [0,1]^{K} \text{ : }  \tilde{p}_j = \int_{s\in \mS_j} \prob(\text{d}s),\text{ } \prob \in \mP \right\}. \label{Eq:ambiguity_reduced}
\end{equation}
If the original scenario set is discrete, the above ambiguity set simplifies to
\begin{equation*}
	\tilde{\mP} \coloneqq \left\{ \tilde{p}\in [0,1]^{K} \text{ : }  \tilde{p}_j = \textstyle\sum_{\{i \mid s_i\in \mS_j\}} p_i,\text{ } p \in \mP \right\}.
\end{equation*}
With this setup, we can write a DRO problem over the reduced ambiguity
set analogously to the original DRO problem~\eqref{Eq:DRO}: 
\begin{align}
	\inf_{x\in\mX} \sup_{\tilde{p} \in \tilde{\mP}}  \E_{\tilde{s}\sim \tilde{p}} \left[f(x,\tilde{s})\right]
	=
	\inf_{x\in\mX} \sup_{\tilde{p} \in \tilde{\mP}}  \sum_{j=1}^{K} \left[f(x,\newscen{j})\right]\tilde{p}_j,  \tag{DROred}\label{Eq:DRO_reduced}	
\end{align}

We now show that under the assumptions of Section~\ref{Sec:AggregatingScen} any solution of the DRO problem with a reduced ambiguity set~\eqref{Eq:DRO_reduced} is an \(\alpha \beta\)-approximation for the original DRO problem~\eqref{Eq:DRO}. 

\begin{theorem}\label{thm:abapprox}
Let \(x^*\) be a solution of problem~\eqref{Eq:DRO} that satisfies Assumptions~\ref{assump:func_assumps}.
Let the reduced DRO problem~\eqref{Eq:DRO_reduced} be constructed with scenarios \(\newscen{j} \in \tS\) for some partition \((\mS_1,\dots,\mS_K)\), such that Assumption~\ref{assump:pos_scenario_sets} is fulfilled.
Then every solution \(\tx\) of~\eqref{Eq:DRO_reduced} is an \(\alpha\beta\)-approximation of the DRO problem~\eqref{Eq:DRO}, i.e.
\begin{equation*}
	\sup_{\prob \in \mP} \E_{s \sim \prob}[f(\tx, s)] \leq \alpha \beta \sup_{\prob \in \mP} \E_{s \sim \prob}[f(x^*, s)].
\end{equation*}
\end{theorem}
\begin{proof}
  As the family of sets \(\mS_{j}\), \(j=1,\ldots,K\) is a partition of \(\mS\), we can write
	\begin{equation*}
		\sup_{\prob \in \mP} \E_{s \sim \prob}[f(\tx, s)]
		= \sup_{\prob \in \mP} \int_{s\in\mS} f(\tx,s) \prob(\text{d}s)
		= \sup_{\prob \in \mP} \sum_{j=1}^{K} \int_{s\in S_j} f(\tx, s)\prob(\text{d}s).
	\end{equation*}
	Due to our assumption of positive homogeneity and because the scenarios \(\scen{i}\) are upper bounded by the aggregated scenarios \(\newscen{j}\), we bound the last term as follows:
	\begin{equation*}
	\begin{aligned}
		\sup_{\prob \in \mP} \sum_{j=1}^{K} \int_{s\in S_j} f(\tx, s)\prob(\text{d}s)
		&\leq \alpha \sup_{\prob \in \mP} \sum_{j=1}^{K}\int_{s\in S_j}f(\tx, \newscen{j})\prob(\text{d}s)\\
		&= \alpha \sup_{\prob \in \mP} \sum_{j=1}^{K}f(\tx, \newscen{j})\left(\int_{s\in S_j}\prob(\text{d}s)\right)\\
		&= \alpha \sup_{\prob \in \mP} \left\{\sum_{j=1}^{K}f(\tx, \newscen{j})\tilde{p}_j \;\Big|\; \tilde{p}_j = \int_{s\in S_j}\prob(\text{d}s) \right\}.
	\end{aligned}
	\end{equation*}
	Since \(\tilde{x}\) is a minimizer for the reduced problem~\eqref{Eq:DRO_reduced} whereas \(x^*\) is a solution to the original DRO problem, we can write
	\begin{equation*}
	\begin{aligned}
		\alpha \sup_{\prob \in \mP} \left\{\sum_{j=1}^{K}f(\tx, \newscen{j})\tilde{p}_j \;\Big|\; \tilde{p}_j = \int_{s\in S_j}\prob(\text{d}s) \right\}
		\leq \alpha \sup_{\prob \in \mP} \left\{\sum_{j=1}^{K}f(x^*, \newscen{j})\tilde{p}_j \;\Big|\;\tilde{p}_j = \int_{s\in S_j}\prob(\text{d}s) \right\}.
	\end{aligned}
	\end{equation*}
	Now we disaggregate the probabilities again as follows
	\begin{equation*}
	\begin{aligned}
		\alpha \sup_{\prob \in \mP} \left\{\sum_{j=1}^{K}f(x^*, \newscen{j})\tilde{p}_j \;\Big|\; \tilde{p}_j = \int_{s\in S_j}\prob(\text{d}s) \right\} 
		&= \alpha \sup_{\prob \in \mP} \sum_{j=1}^{K}f(x^*, \newscen{j})\left(\int_{s\in S_j}\prob(\text{d}s)\right)\\
		&= \alpha \sup_{\prob \in \mP} \sum_{j=1}^{K}\int_{s\in S_j}f(x^*, \newscen{j})\prob(\text{d}s). 
	\end{aligned}
	\end{equation*}
	Finally, we bound the aggregated scenario \(\newscen{j}\) by the original scenario using Assumption~\eqref{assump:scenario_bounds}:
	\begin{equation*}
	\begin{aligned}
		\alpha \sup_{\prob \in \mP} \sum_{j=1}^{K}\int_{s\in S_j}f(x^*, \newscen{j})\prob(\text{d}s) 
		&\leq \alpha \beta \sup_{\prob \in \mP} \sum_{j=1}^{K}\int_{s\in S_j} f(x^*, s)\prob(\text{d}s) \\
		&= \alpha \beta \sup_{\prob \in \mP} \int_{s\in\mS} f(x^*, s)\prob(\text{d}s) \\
		&= \alpha \beta \sup_{\prob \in \mP} \E_{s \sim \prob}[f(x^*, s)].
	\end{aligned}
	\end{equation*}
	This concludes the proof.
\end{proof}

\begin{remark}
	\begin{enumerate}
		\item We have proven the above results for the case of continuous scenarios.
		They directly extend to discrete scenario sets \(\mS\), where integrals are replaced by finite sums. 
		
		\item Our theoretical results do not rely on any structural assumptions on the ambiguity set.
		In particular, the approximation guarantees remain valid for polyhedral, norm-based, or more general nonlinear ambiguity sets.
		
		\item In the case of positively homogeneous functions of degree \(\gamma > 1\), the product \(\alpha \beta\) is replaced by \((\alpha \beta)^{\gamma}\).
	\end{enumerate}
\end{remark}

	
The authors of~\cite{goerigk_scen_reduction} establish a similar approximation bound for scenario reduction in the context of robust optimization with discrete uncertainty sets.
In their proof, they apply less stringent bounds for the representative scenarios compared to our approach in Lemma~\ref{assump:scenario_bounds}.
However, their method leverages the property that, in linear robust optimization with a discrete uncertainty set, the worst-case scenario is always one of the given discrete scenarios, and a robust solution provides protection against the convex hull of these scenarios.
In contrast, our setting is more general than the one considered in \cite{goerigk_scen_reduction} as it can be applied to a wide class of optimization problems satisfying Assumption \ref{assump:func_assumps}.
Our approach involves uncertain probabilities over the scenarios where constraining the worst-case uncertainties is not feasible a priori for a general convex ambiguity set.

An interesting observation is that, although our approach is applied in a distributionally robust optimization setting with uncertain probabilities, the worst-case approximation bound is independent of any probabilistic information.
Consequently, the bound holds for any objective function satisfying Assumption~\ref{assump:func_assumps}, and for any ambiguity sets whose supports satisfy Assumption~\ref{assump:pos_scenario_sets}, provided that the reduced ambiguity set is constructed by aggregating the probabilities of the original one.
This makes the approximation bound very general and applicable to a wide range of settings.
Although one might initially suspect that this bound serves only as an upper limit and is not actually attainable, the following section presents an example in which the bound is achieved in the limit.

\subsection{Sharpness of bound in Theorem~\ref{thm:abapprox}}
The approximation ratio bound of $\alpha \beta$ may appear loose at first, but it is in fact tight.
It is possible to construct an example in which the reduced scenario set attains the approximation ratio $\alpha \beta$.
Consider the following scenario set
\[
S = \left\{
(1, 1+\varepsilon),
(10, 1+\varepsilon),
(10, 1),
(1, 10)
\right\}.
\]
We now construct the following reduced scenario set
\[
\tilde{S} = \left\{
(1, 1+\varepsilon)
\right\}.
\]
Using Lemma~\ref{assump:scenario_bounds} we can identify the values of $\alpha$ and $\beta$ as $10$ and $1$ respectively. 
Thus, we obtain a theoretical approximation guarantee of $\alpha \beta = 10$.

We now construct a DRO problem whose resulting solution attains this approximation guarantee.
To this end, we construct an ambiguity set on the original scenario set $S$. 
We can write the ambiguity set as
\[
\mathcal{P} = \left\{ p \in [0,1]^4 \mid \;p_2 = 1-\delta,\; p_1=p_3=p_4=\frac{\delta}{3} \right\}.
\]
The reduced ambiguity set is given by 
\[
\tilde{\mathcal{P}} = \left\{ p \in [0,1] \mid \;p_1 = 1 \right\}.
\]
We can then write the DRO problem over the original ambiguity set as
\begin{equation}
\begin{aligned}
\min_{x \in X} &\; \max_{p \in \mP}
\; \mathbb{E}\big[ s^\top x \big]\\
\text{s.t.} &\; x_1 + x_2 = 1,\\
&\; x_1, x_2 \geq 0.
\end{aligned}
\end{equation}
The optimal solution of this problem is given by $x^*$ and is equal to $x_1 = 0, x_2 = 1$ with an optimal objective value of \((1-\delta)(1+\varepsilon) + \frac{\delta}{3}(12 + \varepsilon)\) 
The optimal solution for the reduced ambiguity set is denoted by $\tilde{x}$ and is equal to $x_1 = 1, x_2 = 0$. 
However, when you evaluate the worst case value of this solution on the original ambiguity set you obtain a worst case objective value of $(1-\delta)10 + \frac{\delta}{3}12$.

Given this objective value, we can obtain the approximation ratio as
\[
\text{Approx. error} =
\frac{\displaystyle \max_{p \in P} \mathbb{E}\big[ s^\top \tilde{x} \big]}
     {\displaystyle \max_{p \in P} \mathbb{E}\big[ s^\top x^* \big]}
= \frac{(1-\delta)10 + \frac{\delta}{3}12}{(1-\delta)(1+\varepsilon) + \frac{\delta}{3}(12 + \varepsilon)}
\;\;\longrightarrow\;\; 10 \;\text{as } \varepsilon \to 0,\; \delta \rightarrow 0.
\]

Theorem~\ref{thm:abapprox} ensures that a solution obtained from a set of representative scenarios has a bounded approximation quality.
However, since these bounds may be loose, we want to partition the scenario set \(\mS\) to ensure that the representative scenarios achieve the tightest possible bounds.
This issue is addressed in the Section \ref{sec:scenario_partitions}.

\section{Scenario Partitioning}
\label{sec:scenario_partitions}

Theorem~\ref{thm:abapprox} provides a bound on the performance of a solution of the DRO problem with the reduced ambiguity set. 
The bound depends on the parameters \(\alpha\) and \(\beta\), which depend on the partitioning of the scenario set and the corresponding selection of representative scenarios. 
In this section, we aim to assess how the partitioning of the scenario set and the selection of representative scenarios influence the approximation bound.
Furthermore, we discuss strategies for optimizing said bounds.
We begin by exploring how to identify the optimal representative scenario when the whole scenario set is reduced to a single point.

\subsection{Optimal Representative Scenario}
\label{sec:optimal_representative}
\begin{lemma}
\label{lem:single_scenario}
Let a scenario set \(\mS\) be given.
Then, an optimal single representative scenario achieving the approximation quality \(\alpha \beta\) in Theorem~\ref{thm:abapprox} is given by
\(\underline{s}\), where \(\underline{s}_i = \min_{s \in \mS}
s_i\) for all vector components $i=1,\ldots,m$. 
This scenario achieves an approximation quality of 
\[
	\alpha \beta = \max_{i=1,\dots,m}\frac{\overline{s}_i}{\underline{s}_i} \text{, where } \overline{s}_i = \max_{s \in \mS} s_i.
\]
\end{lemma}
\begin{proof}
Our goal is to find a single scenario \(\tilde{s}\) to represent the scenario set $\mS$ which minimizes the product \(\alpha \beta\), where 
\[
\begin{aligned}
s \leq \alpha \tilde{s} \quad\forall s \in \mathcal{S}, \quad
\tilde{s} \leq \beta s \quad \forall s \in \mathcal{S},
\end{aligned}
\]
componentwise.
We can construct an optimization problem for this setup as follows,
\begin{align}
	\min_{\alpha, \beta, \tilde{s}} &\; \alpha \beta \\
	\text{s.t.} &\; s \leq \alpha \tilde{s} \quad \forall s \in \mathcal{S},\\
	&\; \tilde{s} \leq \beta s \quad \forall s \in \mathcal{S},\\
	&\;  \alpha,\beta \geq 0, \tilde{s} \in \R^m.
\end{align}
Let \(\underline{s}\) be such that \(\underline{s}_i = \min_{s \in \mathcal{S}}s_i\) and let \(\overline{s}\) be such that \(\overline{s}_i = \max_{s \in \mathcal{S}} s_i\) for each component $i=1,\ldots,m$.
W.l.o.g, we assume that \(\beta = 1\) by rescaling $\tilde{s}$, otherwise redefine $s'=\beta s$ and $\alpha'=\alpha \beta$.
We simplify the above problem by 
\begin{equation}
	\label{prob:approx_basic}
\begin{aligned}
\min_{\alpha, \tilde{s}} &\; \alpha\\
\text{s.t.} &\; \overline{s} \leq \alpha \tilde{s},\; \tilde{s} \leq \underline{s},\\
&\; \alpha \geq 0, \tilde{s} \in \R^m.
\end{aligned}
\end{equation}
We claim that the optimal solution of the above problem is \(\tilde{s} = \underline{s}\) and \(\alpha = \max_{i=1,\dots,m}\frac{\overline{s}_i}{\underline{s}_i}\). 
We prove this result by contradiction. 
Suppose this is not true. Then there exists a solution \(\tilde{s}'\in\R^m \) such that \(\tilde{s}' \leq \underline{s}\) with at least one component \(j\in\{1,\cdots,m\}\) such that \(\tilde{s}_j' < \underline{s}_j\) (to ensure that \(\tilde{s}' \neq \underline{s}\)).
Given this point, the associated optimal objective is then given by \(\alpha' = \max_{i=1,\dots,m}\frac{\overline{s}_i}{\tilde{s}_i'}\) due to the first constraint.
However, we know that \(\tilde{s}' \leq \underline{s} = \tilde{s}\) due to $\beta=1$ which implies \(\frac{\overline{s}_i}{\tilde{s}_i'} \geq \frac{\overline{s}_i}{\tilde{s}_i}\) for all $i=1,\cdots,m$.
By the same argument, \(\frac{\overline{s}_j}{\tilde{s}_j'} > \frac{\overline{s}_j}{\tilde{s}_j}\) holds because \(\tilde{s}_j' < \underline{s}_j\).
This implies \(\alpha' = \max_{i=1,\dots,m}\frac{\overline{s}_i}{\tilde{s}_i'} > \max_{i=1,\dots,m}\frac{\overline{s}_i}{\tilde{s}_i} = \alpha\).
This contradicts the assertion that \(\tilde{s}'\) is optimal.
Hence \(\tilde{s} = \underline{s}\) is optimal.
\end{proof}

The previous result is generalized such that the representative scenario \(\tilde{s}\) is not limited to \(\underline{s}\) which might not be an element of $\mS$.
We show that any point on the line joining \(\underline{s}\) and \(\overline{s}\) achieves the same approximation bound. 

\begin{corollary}\label{coro:diagonale}
	Given a compact scenario set $\mS$ that fulfills Assumption~\ref{assump:pos_scenario_sets} .  Any point \(s'\) on the line joining \(\underline{s}\) and \(\overline{s}\) achieves an approximation ratio of $\max_{i=1,\dots,m} {\overline{s}_i}/{\underline{s}_i}$.
\end{corollary}
We provide the proof of Corollary~\ref{coro:diagonale} and further insights on the set of optimal representative scenarios for a given partition in the Appendix in Theorem~\ref{theorem:alpha_beta} and~\ref{App:representatives}.
The color map in Figure~\ref{fig:alpha_beta} illustrates the resulting approximation factor $\alpha\beta$ in dependence of the representative scenario $\tilde{s}$ for $\mS=[1,3] \times [1,2] \subset \R^2$.
In this picture, the approximation factors are indicated using
different colors, see the legend.
The region enclosed by the two black lines achieves the best possible approximation factor, here 3.


\begin{figure}[htb]
	\centering
	\includegraphics[width=0.7\linewidth]{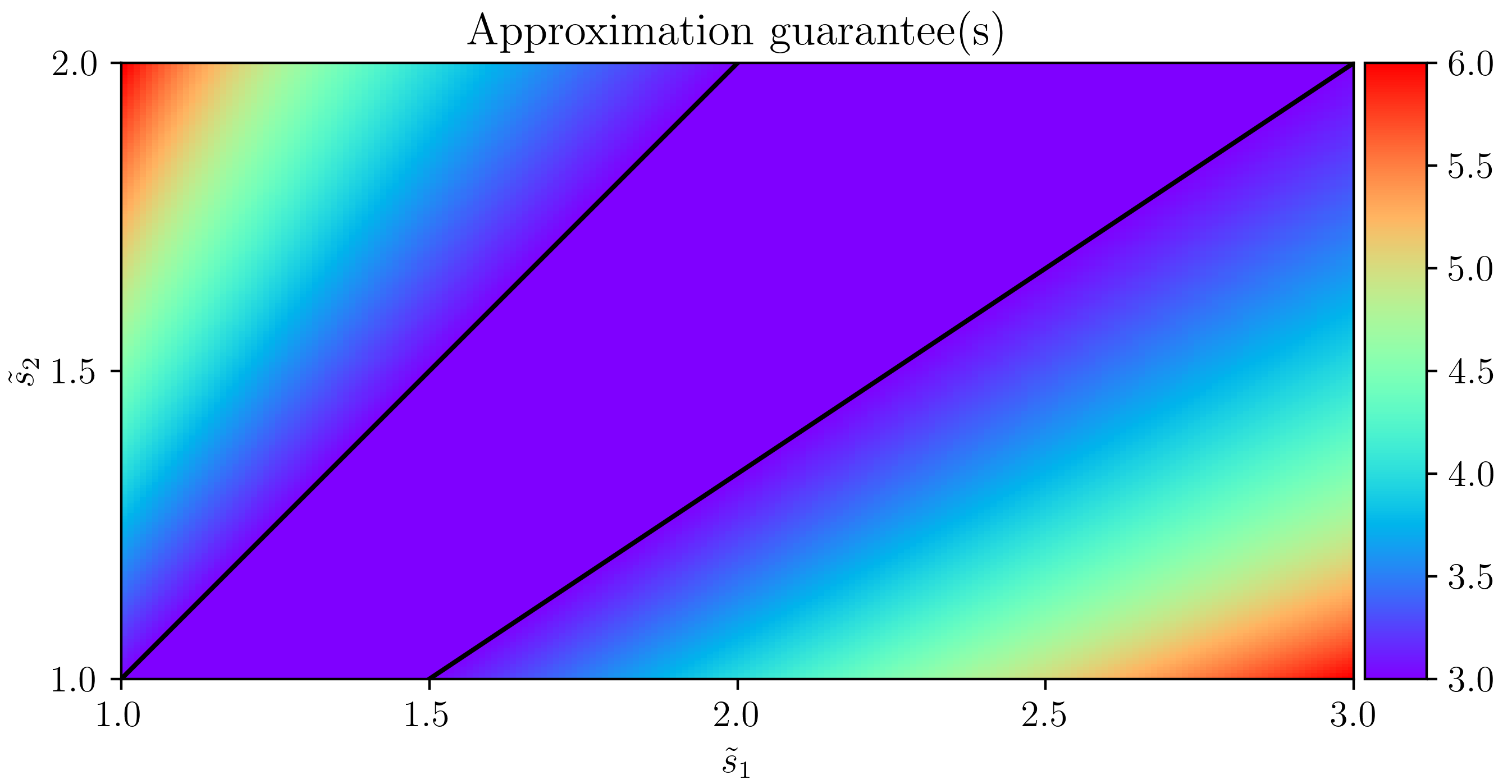}
	\caption{\label{fig:alpha_beta} Approximation guarantee as a function of the representative scenario \(\tilde{s} = (\tilde{s}_1, \tilde{s}_2)\) for scenario set $\mS=[1,3]\times[1,2]\subset \R^2$.}
\end{figure}

So far, we have discussed how to choose a single optimal representative scenario for a set.
Now, we focus on partitioning a set to achieve the best approximation quality over the entire partition. 

\subsection{Optimal Scenario Partitioning}
We begin by modeling the partitioning problem as an optimization problem.
For a representative scenario \(\newscen{}\) of \(\mS\), let the scaling factors \(\alpha\) and \(\beta\) be given by
\[
	\alpha(\newscen{}, \mS) = \max_{i=1,\dots,m}\frac{\max_{s \in \mS} s_i}{\newscen{i}} \;\text{ and }\; \beta(\newscen{}, \mS) = \max_{i=1,\dots,m}\frac{\newscen{i}}{\min_{s \in \mS} s_i}.
\]
The approximation guarantee of \(\newscen{}\) for this set \(\mS\) is given by 
\begin{equation}
t(\newscen{}, \mS) = \alpha(\newscen{}, \mS) \beta(\newscen{}, \mS),
\end{equation}
with the optimal single representative scenario approximation guarantee being computed as
\begin{equation}
	t_1(\mS) = \min_{\newscen{}} \; t(\newscen{}, \mS).
\end{equation}

Our next step is to find an optimal partition of the set \(\mathcal{S}\) into \(K\) subsets $\{\mS_1,\ldots,\mS_K\}$ with \(\mS=\cup_{j=1}^K \mS_j\) and \(\mS_{j}\cap \mS_{k} =\emptyset\) for
all \(j\neq k\), which minimizes the approximation guarantee.
For a given collection of sets, the scaling factors \(\alpha\) and \(\beta\) are the maximum of the scaling factors over the individual sets, see Lemma~\ref{assump:scenario_bounds}.
Therefore, given a partition \(\{\mS_1,\dots,\mS_K\}\) of the set \(\mS\), the optimal approximation guarantee for this partition is 
\[
	t(\mS_1,\dots,\mS_K) = \min_{\newscen{1},\dots,\newscen{K}}\left(\max_{u=1,\dots,K} \;\alpha(\newscen{u}, \mS_u)\right) \left(\max_{v=1,\dots,K} \;\beta(\newscen{v},\mS_v)\right).
\]
Thus, we can formulate an optimization problem to compute the optimal partitioning as 
\begin{align}
	\label{prob:opt_partition}
	t_K(\mS) = \min_{\mS_1,\dots,\mS_K} &\; t(\mS_1,\dots,\mS_K)\\
	\text{s.t.} \hspace{0.3cm} &\; \bigcup_{k=1}^{K} \mS_k = \mathcal{S},\; \mS_j \cap \mS_k = \varnothing \quad \forall j \neq k.
\end{align}

With this setup, we can state the following upper bound on the approximation guarantee, using a covering with hyperrectangles instead of a partionining.
The next theorem provides a tradeoff between the number of clusters $K$ and the approximation guarantee.
\begin{theorem}
	\label{thm:approx_bound}
	Given a scenario set \(\mS \subseteq \mathbb{R}_{+}^m\) whose
        projections along each dimension \(i=1,\dots,m\) are connected
        and a natural number \(K\) such that \(K =
        \prod_{i=1}^{m}r_i\) for some collection of natural numbers
        \(\{r_1,\dots,r_m\}\), the optimal approximation ratio
        \(t_K(S)\) that can be obtained with \(K\) partitions is bounded by
	\[
	t_K(\mS) \leq \max_{i=1,\dots,m}\left(\frac{\overline{s}_i}{\underline{s}_i}\right)^{\frac{1}{r_i}},
	\]
	where \(\overline{s}_i\) and \(\underline{s}_i\) are upper and
        loer bounds of the set \(\mS\) when projected orthogonally on dimension \(i=1,\ldots,m\). 
	This upper bound is attained by a partition of \(\mS\) into hyperrectangles with \(r_i\) partitions of the interval \([\underline{s}_i,\overline{s}_i]\) along dimension \(i=1,\ldots,m\). 
\end{theorem}

To prove Theorem~\ref{thm:approx_bound}, we first find the optimal ratio \(t(\mS)\) when the set \(\mS\) is one dimensional. 


\begin{lemma}
	\label{lem:one_dim_approx_bound}
	Let \(\mS\) be an interval in \(\mathbb{R}_{+}\) such that \(\mS = [\underline{s}, \overline{s}]\) with \(0 < \underline{s} < \overline{s}\). 
	Then, the optimal approximation guarantee of a partition of the set \(\mS\) is bounded from above by
	\[
	t_K(\mS) \leq \left(\frac{\overline{s}}{\underline{s}}\right)^{\frac{1}{K}}.
	\]
\end{lemma}
\begin{proof}
  To obtain an upper bound, we determine 
  the approximation ratio for any choice of partition.
Therefore, let the set \(\mS\) be divided into
\(K\) sub-intervals by the variables \(w_1,\dots,w_{K-1} \in [\underline{s}, \overline{s}]\) as
\begin{equation}
		\{[\underline{s}, w_1), [w_1, w_2), \dots, [w_{K-2}, w_{K-1}), [w_{K-1}, \overline{s}]\}.
\end{equation}
We can write the optimization problem over the partition created by these subintervals as 
\[
	t(\mS_1,\dots,\mS_K) = \min_{\newscen{1},\dots,\newscen{K}}\left(\max_{u=1,\dots,K} \;\frac{\overline{s}_u}{\newscen{u}}\right) \left(\max_{v=1,\dots,K} \;\frac{\newscen{v}}{\underline{s}_v}\right),
\]
where \(\overline{s}_u = w_u\) with \(w_K = \overline{s}\) and \(\underline{s}_v = w_{v-1}\) with \(w_0 = \underline{s}\).
Any feasible scenario yields an upper bound.
So we set \(\newscen{j} = \underline{s}_j\) for all \(j = 1,\dots, K\).
Then we obtain  
\[
	t(\mS_1,\dots,\mS_K) \leq \max_{u=1,\dots,K} \;\frac{\overline{s}_u}{\underline{s}_{u}} = \max_{u=1,\dots,K} \;\frac{w_u}{w_{u-1}}.
\]

This upper bound can be tightened by finding the optimal values of \(w_1,\dots,w_K\) to minimize the bound on \(t(\mS_1,\dots,\mS_K)\) which can be expressed as 
\begin{equation}
	\label{prob:max_single}
	\begin{aligned}
	\min_{w_1,\dots,w_{K-1}} &\; \max\left(\frac{w_1}{\underline{s}}, \frac{w_2}{w_1}, \dots, \frac{w_{K-1}}{w_{K-2}}, \frac{\overline{s}}{w_{K-1}}\right)\\
	\text{s.t.} \;\quad &\; \underline{s} \leq w_1 \leq \cdots \leq w_{K-1} \leq \overline{s}.
	\end{aligned}
\end{equation}
To solve this problem we replace the ratios \(\frac{w_1}{\underline{s}}, \frac{w_2}{w_1}, \dots\) by their logarithm.
This is valid since the logarithm function is monotonically increasing, all elements are positive, and \(\log(\max{a_1,\dots,a_n}) = \max(\log a_1, \dots, \log a_n)\).
We can write the objective function of the problem as
\begin{equation*}
	\begin{aligned}
	\max\left(\log \frac{w_1}{\underline{s}}, \log \frac{w_2}{w_1}, \dots, \log \frac{w_{K-1}}{w_{K-2}}, \log \frac{\overline{s}}{w_{K-1}} \right).
	\end{aligned}
\end{equation*}
This is equivalent to $\max\left(\log w_1 - \log \underline{s}, \log w_2 - \log w_1, \dots, \log \overline{s} - \log w_{K-1}\right)$.
Let \(\hat{w}_i = \log w_i\). 
Then, the objective can be expressed as 
\begin{equation*}
	\begin{aligned}
	\max\left(\hat{w}_1 - \log{\underline{s}}, \hat{w}_2 - \hat{w}_1, \dots, \log{\overline{s}} - \hat{w}_{K-1}\right).
	\end{aligned}
\end{equation*}
Define \(\delta_i = \hat{w}_i - \hat{w}_{i-1}\) with \(\hat{w}_0 = \log \underline{s}\) and \(\hat{w}_{K} = \log \overline{s}\).
Then the above problem is $\max\left(\delta_1, \delta_2, \dots, \delta_K\right)$.
The constraints can equivalently be expressed in terms of \(\delta_i\) as follows:
\begin{equation*}
	w_i \geq w_{i-1} \iff \delta_i \geq 0, \quad \text{and} \quad 	\underline{s} \leq w_1 \leq \dots w_{K-1} \leq \overline{s} \iff \sum_{i=1}^{K}\delta_i = \log \frac{\overline{s}}{\underline{s}} .
\end{equation*}
Thus, the problem can be written as
\begin{equation*}
\begin{aligned}
	\min_{\delta_i} \; \max\left(\delta_1, \delta_2, \dots, \delta_K\right)
	\;\text{s.t.} \; \sum_{i=1}^{K}\delta_i = \log \frac{\overline{s}}{\underline{s}},\; \delta_i \geq 0.
\end{aligned}
\end{equation*}
The above problem is permutation invariant and the maximization is performed over linear functions. 
Hence, all \(\delta\) variables are equal at optimality as exemplarily displayed in Figure~\ref{fig:log_scaling} for $K=4$.
Thus \(\delta_i^{*} = \frac{1}{K}\log \frac{\overline{s}}{\underline{s}} \) leads to an optimal objective of \(\frac{1}{K}\log \frac{\overline{s}}{\underline{s}}\).
Therefore, the optimal objective in terms of the original problem is given by \(\exp(\delta_i^*)\) for any \(i\).
This allows us to write
\[
t_K(\mS) \leq \exp \left(\frac{1}{K}\log \frac{\overline{s}}{\underline{s}} \right) = \left(\frac{\overline{s}}{\underline{s}}\right)^{\frac{1}{K}},
\]
which concludes the proof.
\end{proof}

\begin{figure}[htb]
	\centering
	\begin{tikzpicture}
		\draw[-] (1,1) -- (16,1);
		\foreach \x [count=\i from 0] in {$a$, $w_1$, $w_2$, $w_3$, $b$} { 
				\draw[thick] (2^\i,0.9) -- (2^\i,1.1);
				\node[below] at (2^\i,0.9) {\x};
			}
			
			\draw[-] (1,0) -- (16,0);
			\foreach \x [count=\n from 0] in {$\log a$, $\hat{w}_1$, $\hat{w}_2$, $\hat{w}_3$, $\log b$} {
				\pgfmathsetmacro\pos{1+3.75*\n} 
				\draw[thick] (\pos,-0.1) -- (\pos,0.1);
				\node[below] at (\pos,-0.1) {\x};
			}
			\foreach \x in {1,2,4,8} {
				\pgfmathsetmacro\pos1+ln(\x)/ln(10)*10} 
		
	\end{tikzpicture}
	\caption{\label{fig:log_scaling} Variable substitution in proof of Lemma~\ref{lem:one_dim_approx_bound} exemplarily for $K=4$.}
\end{figure}
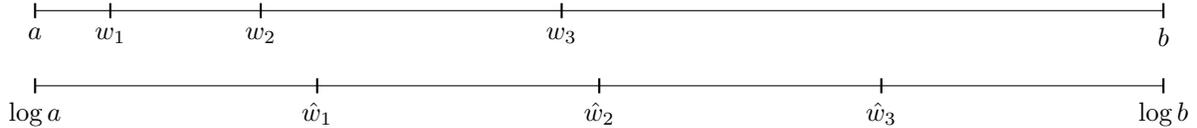
We proceed by providing a proof for Theorem~\ref{thm:approx_bound}, which leverages Lemma~\ref{lem:one_dim_approx_bound} along each dimension.

\begin{proof}
(\emph{Theorem~\ref{thm:approx_bound}}):
We obtain an upper bound for the optimal objective value of problem~\eqref{prob:opt_partition} by evaluating its value for a suitable partition.
To this end, we cover the set \(\mS\) with \(K\) hyperrectangles as in Figure~\ref{fig:set_partition}.
These hyperrectangles are constructed by partitioning the interval \([\underline{s}_i,\; \overline{s}_i]\) along each dimension \(i=1,\ldots,m\) into \(r_i\) subintervals such that each subinterval contains its left endpoint, but not its right endpoint, except the final subinterval, which contains \(\overline{s}_i\).
Given a partition \(S_1,\dots,S_K\) constructed by these hyperrectangles, we bound the objective value \(t(\mS_1,\dots,\mS_K)\).
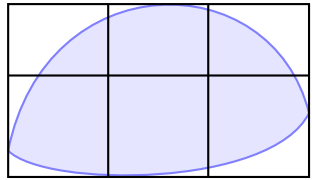
\begin{figure}[htb]
	\centering
	\begin{tikzpicture}
		\filldraw[fill=blue!20, draw=blue, thick, opacity=0.5] 
			(1,1) .. controls (1.5,3.5) and (4.5,3.5) .. (5,1.5)
			.. controls (4.5,0.5) and (1.5,0.5) .. (1,1);
		
		\draw[thick] (1,0.65) rectangle (5,2.95); 
		\draw[thick] (2.33,0.65) -- (2.33,2.95); 
		\draw[thick] (3.66,0.65) -- (3.66,2.95); 
		\draw[thick] (1,2) -- (5,2);           
		\end{tikzpicture}
	\caption{\label{fig:set_partition}Covering of a scenario set (blue region) with hyperrectangles where \(r_1 = 3\) and \(r_2 = 2\) for \(K = 6\).}
\end{figure}
	
By definition, we have 
\[
	t(\mS_1,\dots,\mS_K) = \min_{\newscen{1},\dots,\newscen{K}}\left(\max_{u=1,\dots,K} \max_{i = 1,\dots,m}\;\frac{\overline{s}_{ui}}{\newscen{ui}}\right) \left(\max_{v=1,\dots,K} \max_{l=1,\dots,m}\;\frac{\newscen{vl}}{\underline{s}_{vl}}\right),
\]
We obtain an upper bound by selecting \(\newscen{j} = \underline{s}_j\) for all hyperrectangles \(j=1,\dots,K\). 
This implies \(\newscen{vl} = \underline{s}_{vl}\) for all subsets \(v = 1,\dots,K,\) and all dimensions \(l=1,\dots,m\).
Thus, we write 
\[
	t(\mS_1,\dots,\mS_K) \leq \max_{u=1,\dots,K} \max_{i = 1,\dots,m}\;\frac{\overline{s}_{ui}}{\underline{s}_{ui}}.
\]
Now, we can exchange the two max terms and obtain 
\[
	t(\mS_1,\dots,\mS_K) \leq \max_{i = 1,\dots,m} \max_{u=1,\dots,K}\;\frac{\overline{s}_{ui}}{\underline{s}_{ui}}.
\]
While there are a total of \(K\) partitions along any particular dimension \(i\), there are only \(r_i\) effective partitions by construction in each dimension $i=1,\ldots,m$.
Equivalently, we state that for a fixed \(i\), there are only \(r_i+1\) distinct values of the ratio \(\frac{\overline{s}_{ui}}{\underline{s}_{ui}}\) over all \(u=1,\dots,K\).
Subsequently, we replace the index \(u\) with a new index \(w \in W_i\) with \(W_i = \{1,\ldots,r_i\}\), which iterates over the values in the index set \( W_i\) for all $i=1,\cdots,m$.
Then we obtain
\[
	t(\mS_1,\dots,\mS_K) \leq \max_{i = 1,\dots,m} \max_{w \in W_i}\;\frac{\overline{s}_{wi}}{\underline{s}_{wi}}.
\]

From Lemma~\ref{lem:one_dim_approx_bound}, we know that given a partition of size \(r_i\), the approximation ratio along any single dimension can be bounded above by \(
	\left(\frac{\overline{s}_{i}}{\underline{s}_{i}}\right)^{\frac{1}{r_i}}\).
	Therefore, we can write 
	\[
		t(\mS_1,\dots,\mS_K) \leq \max_{i = 1,\dots,m} \;\left(\frac{\overline{s}_{i}}{\underline{s}_{i}}\right)^{\frac{1}{r_i}}.
	\]
Let \(\overline{s}_i\) and \(\underline{s}_i\) be the upper and lower bounds of the projection of set \(\mS\) along dimension \(i\) and let \(r_i = \left(\frac{\overline{s}_i}{\underline{s}_i}\right)^{\frac{1}{r_i}}\).
Since this was for a partition constructed using hyperrectangles, we get
\[
		t_K(\mS) \leq \max_{i = 1,\dots,m} \;\left(\frac{\overline{s}_{i}}{\underline{s}_{i}}\right)^{\frac{1}{r_i}}.
	\]
This concludes the proof.

\end{proof}

The previous results discussed the issue of partitioning scenario sets. 
However, we made the key assumption that the projections of the set along each dimension are connected.
This assumption is similar to but weaker than requiring the scenario sets to be convex. 
It is needed because we want to prove a bound on the maximum component-wise ratio of two uncertain parameters in a set.  

However, proving such a bound for any arbitrary set can be quite challenging. 

As such, we prove that an upper bound for this ratio can be obtained by estimating it for an axis-parallel hyperrectangle containing the set.
Computing this ratio for a hyperrectangle is simpler. 
We can compute the maximum ratio for each axis independently and then take the maximum across all axes. 

The assumption allows us to construct this hyperrectangle by taking the product of the projections of the set along each dimension. 

If the projections are not connected, we do not obtain a single connected hyperrectangle for which to compute this ratio. 

In such a setting, however, our results remain meaningful, as we can focus on a single hyperrectangle containing all the pieces of a set.
Then we obtain a more conservative upper bound of the ratio.

The following section focuses on computing a partition for discrete scenarios.


\subsection{Optimal Partitioning for Discrete Scenarios}\label{Sec:Discrete_Scenarios}
Our previous results provide an upper bound on the best approximation ratio achievable for Theorem~\ref{thm:abapprox} when partitioning a set.
However, the calculated bound is not necessarily tight.
While calculating the optimal partition is challenging for an
arbitrary scenario set, the problem becomes computationally easier if we limit ourselves to random variables with discrete and finite support.
We next construct a mixed integer program to optimally partition a discrete scenario set and compute the optimal representative scenarios, which minimize the approximation guarantee \(\alpha \beta\) for Theorem~\ref{thm:abapprox}.

To calculate \(\alpha\) and \(\beta\), we directly use the inequalities from Lemma~\ref{assump:scenario_bounds} as componentwise constraints in the partitioning problem. 
Given a collection of discrete scenarios $\mS=\{s_1,\ldots,s_{|\mS|}\}$, the optimization problem for the computation of the reduced scenarios $\tS=\{\tilde{s}_1,\ldots,\tilde{s}_K\}$ while minimizing the approximation guarantee \(\alpha\beta\) given by Theorem~\ref{thm:abapprox}, can be formulated as
\begin{subequations} \label{Eq:Clustering_MINLP}
	\begin{align}
		\min_{\alpha, \beta,\tilde{s},z} \quad & \alpha\beta \label{new_MINLPa}\\
		\text{ s.t. } \hspace{0.3cm} &  s_i \leq  \alpha \tilde{s}_j + M(1-z_{ij})\quad \forall i=1,\ldots,|\mS|, j=1,\ldots,K, \label{new_MINLPb}
		\\
		& \tilde{s}_j \leq \beta s_i + M(1-z_{ij}) \quad \forall i=1,\ldots,|\mS|, j=1,\ldots,K, \label{new_MINLPc}	
		\\
		&\sum_{j=1}^{K} z_{ij} = 1 \quad \forall i=1,\ldots,|\mS|, \label{new_MINLPd}
		\\
		& \sum_{i=1}^{|\mS|}z_{ij} \geq 1 \quad \forall j=1,\ldots,K, \label{new_MINLPe}
		\\
		& \alpha,\beta\geq 0, \quad  z_{ij} \in \{0,1\} \quad \forall i,j.\label{new_MINLPf}
	\end{align}
\end{subequations}
In the problem described above, the binary variable \(z_{ij} = 1\) indicates that the original scenario \(s_i\) is assigned to the representative scenario \(\tilde{s}_j\) of cluster \(\tilde{S}_j\).  
The constraints~\eqref{new_MINLPb} and~\eqref{new_MINLPc} enforce the inequalities required for Theorem~\ref{thm:abapprox}.  
The included big-\(M\) terms \(M(1-z_{ij})\), where the real number \(M > 0\) is
sufficiently large, ensure that the constraints are enforced only when
scenario \(s_i\) is assigned to cluster \(\tilde{S}_j\) and redundant otherwise.  
Equations~\eqref{new_MINLPd} and inequalities~\eqref{new_MINLPe} enforce, respectively, that each scenario is assigned to exactly one cluster and that all clusters are non-empty.

In order to linearize the quadratic mixed-integer nonlinear program~\eqref{Eq:Clustering_MINLP} we can set w.l.o.g. \(\beta=1\) (or \(\alpha=1\)), since for a solution of~\eqref{Eq:Clustering_MINLP}, one can set \(\hat{s}\coloneqq \tilde{s}_j/\beta\). \(\hat{s}\) is also a solution of~\eqref{Eq:Clustering_MINLP} with \(\hat{\alpha}\coloneqq \alpha \beta\) and \(\hat{\beta}=1\).
Then, problem~\eqref{Eq:Clustering_MINLP} can be written as the mixed-integer linear program (with \(t\coloneqq 1/\alpha\))
\begin{subequations} \label{Eq:Clustering_MIP}
	\begin{align}
		\max_{t,\tilde{s},z} \quad & t \label{Eq:Clustering_MIPa}\\
		\text{ s.t.} \hspace{0.3cm} & t s_i \leq  \tilde{s}_j + M(1-z_{ij})\quad \forall i=1,\ldots,|\mS|, j=1,\ldots,K, \label{Eq:Clustering_MIPb}
		\\
		& \tilde{s}_j \leq  s_i + M(1-z_{ij}) \quad \forall i=1,\ldots,|\mS|, j=1,\ldots,K, \label{Eq:Clustering_MIPc}	
		\\
		&\sum_{j=1}^{K} z_{ij} = 1 \quad \forall i=1,\ldots,|\mS|, \label{Eq:Clustering_MIPd}
		\\
		& \sum_{i=1}^{|\mS|}z_{ij} \geq 1 \quad \forall j=1,\ldots,K, \label{Eq:Clustering_MIPe}
		\\
		& t \geq 0, \quad  z_{ij} \in \{0,1\} \quad \forall i,j.\label{Eq:Clustering_MIPf}
	\end{align}
\end{subequations}

The big-M constant for the constraints can be calculated independently
for every constraint and each component $k=1,\ldots,m$:
\begin{align*}
	\eqref{Eq:Clustering_MIPb}: M_k= s_{ik}, \quad
	~\eqref{Eq:Clustering_MIPc}: M_k= \max_{i=1,\ldots,|\mS|}\{s_{ik}\} - s_{ik}.
\end{align*}
In our experiments, the resulting big-\(M\) values remain moderate in size and so that the MIP solution remains numerically stable, 
as all scenario values are strictly positive and bounded (cf.~Assumption~\ref{assump:pos_scenario_sets}).

To summarize, in Section~\ref{sec:scenario_partitions} we have
developed an approach for optimal scenario reduction for the worst-case approximation guarantee $\alpha\beta$ as in Theorem~\ref{thm:abapprox}.
In order to construct the approximating DRO problem~\eqref{Eq:DRO_reduced}, we discuss next how to construct a new ambiguity set~$\tP$ over the reduced scenarios based on the original ambiguity set~$\mP$.
This is addressed in Section~\ref{sec:dimen_red_amb_sets}.



\section{Dimension Reduction of Ambiguity Sets}
\label{sec:dimen_red_amb_sets}
So far we have reduced the scenario set to provide the tightest possible approximation guarantee \(\alpha \beta\) of the reduced DRO problem~\eqref{Eq:DRO_reduced}.
We now proceed by discussing the structure of the ambiguity set~\(\tilde{\mP}\) over the reduced scenarios given by~\eqref{Eq:ambiguity_reduced} for a given original ambiguity set \(\mP\). 
This discussion is limited to the case of finitely many scenarios where we can derive explicit expressions of the reduced ambiguity set for confidence intervals and ellipsoidal ambiguity sets. 

For a given clustering $\{\mS_1,\ldots,\mS_K\}$ with \(\mS=\bigcup_{j=1}^{K}{\mS}_j\), \({\mS}_k \cap {\mS}_l = \emptyset\) for $k \neq l$, we construct the reduced ambiguity set as
\begin{equation}
	\tilde{\mP}= \left\{\tilde{p}\in[0,1]^{K} : \tilde{p}_j = \textstyle\sum_{\{ i\text{:}s_i \in \mS_j\}} p_i,\; p\in\mP\right\} = \left\{\tilde{p}\in[0,1]^{K} :\tilde{p} = Ap,\; p\in\mP\right\}, \label{Eq:reduced_ambiguity}
\end{equation}
where $A\in\{0,1\}^{K \times \vert \mS \vert}$ is a matrix mapping a scenario $s_i \in \mS$ to a cluster $\tilde{S}_j$ via the binary matrix entry $a_{ij}\in\{0,1\}$.
We note that the matrix $A$ has full rank, i.e., $rank(A)=K$.
In this section, we focus on ambiguity sets that maintain their geometrical structure under the linear transformation induced by \(A\).


%
%
%

\subsection{Interval Ambiguity Sets}\label{Sec:Confidence_Intervalls}

Let the original ambiguity set \(\mP\) consist of intervals, i.e. we can write
\begin{equation}
	\mP=\left\{p\in [0,1]^{|\mS|} : l \leq p \leq u,\; \textstyle\sum_{i=1}^{|\mS|} p_i=1\right\}. \label{Eq:box_ambiguity}
\end{equation}
For this ambiguity set, we prove the following result.
\begin{lemma}
	If the original ambiguity set is given by~\eqref{Eq:box_ambiguity}, then the reduced ambiguity set~\eqref{Eq:reduced_ambiguity} can be written as
	\begin{equation}
		\label{def:reduced_ambiguity_set}
		\begin{aligned}
			\tilde{\mP}
			=& \left\{\tilde{p} \in  [0,1]^{K}  : Al \leq \tilde{p} \leq Au,\; \textstyle\sum_{j=1}^{K}  \tilde{p}_j=1\right\},
		\end{aligned}
	\end{equation}
	where $A \in \{0, 1\}^{K \times \vert \mS \vert}$ is the matrix that maps the scenarios to their respective clusters.
\end{lemma}
\begin{proof} By aggregating the probabilities for the clustered scenarios according to matrix~$A$, we obtain:
	\begin{equation*}
		\tilde{\mP}
		= \left\{\tilde{p}\in[0,1]^{K} :\tilde{p} = Ap,\; p\in\mP\right\}
		= \left\{\tilde{p} \in  [0,1]^{K}  : Al \leq \tilde{p} \leq Au,\; \textstyle\sum_{j=1}^{K}  \tilde{p}_j=1.\right\}
	\end{equation*}
	This concludes the proof.
\end{proof}

This result can also be extended to general polytopes.
Specifically, if the original ambiguity set is a polytope, then the reduced set is also a polytope.
This holds due to scenario reduction being equivalent to a linear transformation.
To obtain the reduced polytope, we can either use the lifted formulation or develop a simplified formulation using Fourier-Motzkin elimination applied to~\eqref{Eq:reduced_ambiguity}.
We do not go into further detail here.

\subsubsection{Ellipsoidal ambiguity sets}\label{Sec:ellipsoidal}
Another very popular modeling approach for the ambiguity set is the ellipsoidal ambiguity set. 
We can show that the reduced ambiguity set~$\tP$ is ellipsoidal if the original ambiguity set has this geometry.
For the remainder of this section, let our original ambiguity set be given by 
\[
\mP=\left\{p\in [0,1]^{|\mS|} : (p - p^0)^{\top}\Sigma^{-1}(p - p^0) \leq r^2,\;\; \textstyle\sum_{i=1}^{|\mS|} p_i=1\right\},
\]
where $\Sigma \in \mathbb{R}^{\vert \mS \vert \times \vert \mS \vert}$ is a symmetric positive definite matrix, $p^0 \in [0,1]^{\vert \mS \vert}$ and $r > 0$.

\begin{restatable}{lemma}{lemredellips}
	Let the original ambiguity set \(\mP\) be an ellipsoid on the probability simplex. Then the ambiguity set over the reduced scenarios is also an ellipsoid and is given by 
	\begin{equation*}
		\tilde{\mP}=\left\{\tilde{p}\in [0,1]^{K} : (\tilde{p} - A p^0)^{\top}\tilde{\Sigma}^{-1}(\tilde{p} - A p^0) \leq r^2,\;\; \textstyle\sum_{i=1}^{K} \tilde{p}_i=1\right\},
	\end{equation*}
	where \(\tilde{\Sigma} = A \Sigma A^{\top}\).
\end{restatable}

For completeness, the proof is provided in the appendix. These two
classes of ambiguity sets are used in the computational results.

\section{Extension to Convex Quadratic Objectives}
\label{sec:extensions}
So far we have discussed the problem of scenario reduction with
objectives that that satisfy Assumption~\ref{assump:func_assumps}.
Due to the relevance of quadratic objectives, we focus our attention
on them next. 
We consider a convex quadratic objective 
	$f(x,Q) = x^\top Q x$,
with a symmetric positive definite matrix \(Q\in\R^{n\times n}\).
Similar to Assumption~\ref{assump:func_assumps}, we make use of the fact that for convex quadratic objectives we have the following implication,
\begin{equation*}
	f(x,Q) \leq \alpha f(x,\tilde{Q}) \quad \forall Q,\tilde{Q} \in \R^{n \times n} \text{ with } Q\preceq \alpha\tilde{Q},
\end{equation*}  
where $\preceq$ denotes the partial order defined by the convex cone of positive semidefinite matrices.

We model the stochasticity in the objective by assuming that each
scenario corresponds to a positive definite matrix \(Q_i\in\R^{n\times n}\) for \(i=1,\ldots,|\mS|\).
Given these scenarios, we state the following optimization problem to identify the set of reduced scenarios. 

\begin{theorem}
	The optimization problem for computing the matrix clusters $S_j$ and its representatives $\tilde{Q}_j$ that minimize the approximation guarantee \(\alpha\beta\) for the reduced DRO problem~\eqref{Eq:reduced_ambiguity} can be formulated as
	\begin{equation*}
		\begin{aligned}
			\min_{\alpha, \beta,\tilde{Q}_j,z} \quad & \alpha\beta \\
			\text{ s.t. } \hspace{0.45cm} &  Q_i \preceq  \alpha \tilde{Q}_j + M(1-z_{ij})\quad \forall i=1,\ldots,|\mS|, j=1,\ldots,K, 
			\\
			& \tilde{Q}_j \preceq \beta Q_i + M(1-z_{ij}) \quad \forall i=1,\ldots,|\mS|, j=1,\ldots,K, 	
			\\
			&\sum_{j=1}^{K} z_{ij} = 1 \quad \forall i=1,\ldots,|\mS|, 
			\\
			& \sum_{i=1}^{|\mS|}z_{ij} \geq 1 \quad \forall j=1,\ldots,K,
			\\
			& \alpha,\beta\geq 0,Q_j\succeq 0 \text{ }\forall j=1,\ldots,K,  z_{ij} \in \{0,1\}\text{ } \forall i,j.
		\end{aligned}
	\end{equation*}
\end{theorem}
\begin{proof}
The reduced scenarios satisfy worst-case approximation guarantees if:
\begin{enumerate}
	\item There exists an \(\alpha>0\), such that for every cluster \(S_j\) we have \(Q_i  \preceq \alpha \tilde{Q}_j \quad \forall Q_i\in S_j\),
	\item There exists a \(\beta>0\), such that for every cluster \(S_j\) we have \(\tilde{Q}_j \preceq \beta Q_i\quad \forall Q_i \in S_j\).
\end{enumerate}
The proof proceeds analogously to the proof of Theorem~\ref{thm:abapprox}, since
\begin{align*}
	Q_i  \preceq \alpha \tilde{Q}_j &\Leftrightarrow  \alpha \tilde{Q}_j -  Q_i  \text{ is positive semidefinite}
	\\
	&\Leftrightarrow	x^\top (\alpha \tilde{Q}_j -  Q_i) x \geq 0 \quad \forall x\in \R^n 	
	\\
	&\Leftrightarrow	x^\top Q_i x \leq \alpha x^\top \tilde{Q}_j x  \quad \forall x\in \R^n.
\end{align*}
We skip further details here.
\end{proof}

Analogously to Section~\ref{Sec:Discrete_Scenarios}, we can set w.l.o.g $\beta=1$ and rewrite the optimization problem.
The optimization problem for the computation of the clusters \(\tilde{s}_j\) with \(z_{ij}\) denoting the membership of scenario $i$ to cluster $j$, and \(t\coloneqq \frac{1}{\alpha}\) can be formulated as
\begin{subequations}\label{Eq:MISDP}
	\begin{align}
		\max_{t,\tilde{Q}_j,z} \quad & t \label{Eq:MISDPa}\\
		\text{ s.t. } \hspace{0.3cm} &  tQ^i \preceq  \tilde{Q}_j + M_i^{(1)}(1-z_{ij})\quad \forall i=1,\ldots,|\mS|, j=1,\ldots,K, \label{Eq:MISDPb}
		\\
		& \tilde{Q}_j \preceq  Q_i + M_i^{(2)}(1-z_{ij}) \quad \forall i=1,\ldots,|\mS|, j=1,\ldots,K, 	\label{Eq:MISDPc}
		\\
		&\sum_{j=1}^{k} z_{ij} = 1 \quad \forall i=1,\ldots,|\mS|, \label{Eq:MISDPd}
		\\
		& \sum_{i=1}^{K}z_{ij} \geq 1 \quad \forall j=1,\ldots,K, \label{Eq:MISDPe}
		\\
		& t\geq 0,Q_j\succeq 0 \text{ }\forall j=1,\ldots,K,  z_{ij} \in \{0,1\}\text{ } \forall i,j. \label{Eq:MISDPf}
	\end{align}
\end{subequations}
The matrix clustering problem is a mixed-integer semidefinite program (MISDP) that can be solved by modern software tools such as the SCIP-SDP plugin for SCIP~\cite{SCIP}.
Valid values for constants $M_i^{(1)}$, $M_i^{(2)}$ for $i=1, \ldots, \vert \mS \vert$ in constraints~\eqref{Eq:MISDPb} and~\eqref{Eq:MISDPc} are given by
\begin{equation}
	M^{(1)}_i= \lambda^{\text{max}}(Q_i), \quad\quad
	M^{(2)}_i= \max_{i=1,\ldots,|\mS|}\{\lambda^{\text{max}}(Q_i)\} - \lambda^{\text{min}}(Q_i),
\end{equation}
where \(\lambda^{\text{max}}(Q)\) and \(\lambda^{\text{min}}(Q)\) denote the largest and smallest eigenvalues of matrix $Q$, respectively.

Solving optimization problem~\eqref{Eq:MISDP} can be time-consuming, especially in high dimensions and for a large number of scenarios.
As a heuristic, we can perform clustering by the $k$-means algorithm with a suitable matrix norm, e.g., the Frobenius norm.
For a given clustering, we can compute an upper bound for the resulting approximation guarantee
\begin{equation}
	\alpha_j = \max_{\{i\colon Q_i \in \mS_{j}\}}\frac{ \lambda^{\text{max}}(Q_i) }{ \lambda^{\text{min}}(\tilde{Q}_i)  }, \quad
	\beta_j = \max_{\{i\colon \tilde{Q}_{i} \in \mS_{j}\}}\frac{ \lambda^{\text{max}}(\tilde{Q}_{j}) }{ \lambda^{\text{min}}(Q_{i})  }, \label{eq:eigvals}
\end{equation}
for each cluster \(j=1,\dots,K\).
The overall worst-case approximation guarantee is given by \(\max_j\alpha_j \beta_j\).
We use the fact here, that \(\lambda^{\text{max}}(A)\leq\lambda^{\text{min}}(B)\) implies \(A\preceq B\) for symmetric matrices.
We also note that this is only an upper bound due to this sufficient criterion.

By combining the results from Sections~\ref{Sec:Discrete_Scenarios} and~\ref{sec:extensions}, one can also further extend our methodology to reduce scenarios for uncertain objective functions of the form $f(x,s)=x^\top Q x + s^\top x$.

While our mixed-integer formulations (MIP for linear objectives and MISDP for quadratic objectives) allow for computing optimally reduced scenario sets with formal approximation guarantees, they are not designed for large-scale use. Solving them to optimality becomes computationally demanding as the number of scenarios and the problem dimension grow. However, this exact formulation serves an important purpose: it provides a benchmark to assess the quality of heuristic scenario reduction methods such as \emph{k-means}. In practice, we recommend using such heuristics for scalability, while relying on the optimal formulation primarily for small to medium-sized problems or when a certified approximation guarantee is needed.



\section{Computational Experiments}
\label{sec:experiments}

In this section, we evaluate the performance of our methods through numerical experiments on distributionally robust mixed-integer linear and quadratic optimization problems.
To highlight the broad applicability of our approach, we extract linear instances from the well-known MIPLIB benchmark library~\cite{furini2018online}.
For the quadratic case, we solve a robust portfolio optimization problem, utilizing real-world financial data.
We provide further details on how we construct the DRO instances for the two testcases in Sections~\ref{Section:MIPLIB} and~\ref{sec:port_opt}.

Given a DRO instance, we reduce the problem size by solving the clustering MIP~\eqref{Eq:Clustering_MIP} for linear objectives and the MISDP~\eqref{Eq:MISDP} for quadratic objectives.
We benchmark our approach against the well-known \emph{k-means} algorithm~\cite{kmeans_macqueen}, using the Euclidean norm for vectors and the Frobenius norm for matrices. 
The selection of representative scenarios for the clusters is not unique, as discussed in Section~\ref{sec:optimal_representative}.
For mixed-integer linear test cases, we select as cluster representatives, the mean of all scenarios within a cluster, projected onto the diagonal of the \( d \)-dimensional hyperrectangle defined by the component-wise minimum and maximum of all scenarios in the cluster.
This choice is valid since the approximation guarantee holds for all points on the diagonal, as shown in Corollary~\ref{coro:diagonale}.
For matrix clustering, the cluster representatives are determined by solving~\eqref{Eq:MISDP} with \(\beta=1\).
In the case of \emph{k-means}, the representatives are directly obtained by the algorithm as the mean of all scenarios within a cluster.

The primary goal of our approach is to improve the tractability of distributionally robust optimization (DRO) problems by reducing the number of scenarios.
To evaluate its effectiveness, we use three performance measures that capture different aspects of the trade-off between computational efficiency and solution quality.

First, the \emph{time factor} (TF) quantifies the speed-up achieved by our scenario reduction method.
It is defined as the ratio between the time required to solve the reduced DRO problem and the time needed for the original problem without reduction.
A TF significantly smaller than one indicates substantial runtime improvements, which is especially relevant for large-scale applications.
Second, the \emph{approximation factor} (AF) captures the impact of scenario reduction on solution quality.
It is defined as the ratio of the worst-case objective value over the original scenarios obtained from the reduced DRO problem to the one from the original problem.
An AF close to one implies that the reduction has little effect on the solution quality, while larger values indicate a loss in objective value due to reduced uncertainty representation.
Third, we introduce the \emph{scenario reduction factor} (SRF), which measures the compression ratio of the scenario set.
It is defined as the size of the original scenario set divided by the size of the reduced set.
Higher values of SRF correspond to a larger reduction in prolems size.

Formally, these performance metrics are given by:
\begin{align}
	\text{Time Factor (TF)} &= \frac{\text{Time to solve \eqref{Eq:DRO_reduced} with reduced scenario set}}{\text{Time to solve \eqref{Eq:DRO} with original scenario set}},\label{Eq:TF}\\
	\text{Approximation Factor (AF)} &= \frac{\text{Worst-case obj. val. of \eqref{Eq:DRO_reduced} with reduced scenario set}}{\text{Worst-case obj. val. of~\eqref{Eq:DRO} with original scenario set}},\label{Eq:AF}\\
	\text{Scenario Reduction Factor (SRF)} &= \frac{\text{Size of original scenario set}}{\text{Size of reduced scenario set}} = \frac{|\mS|}{|\tS|}.\label{Eq:SRF}
\end{align}
In Section~\ref{sec:port_opt}, we also conduct an out-of-sample evaluation to assess the generalization quality of the reduced models.
Here, we compute the objective value using unseen data (e.g., future years in the stock return time series) and compare it to the in-sample performance.

From now on, we refer to instances where scenarios are clustered by solving the MIP~\eqref{Eq:Clustering_MIP} or the MISDP~\eqref{Eq:MISDP} as \emph{opt}.
We calculate these factors for both scenario reduction methods, i.e., \emph{k-means} and \emph{opt}.

\subsection{MIPLIB instances}\label{Section:MIPLIB}
All computations in this section were performed using a Python implementation on machines equipped with an Intel® Xeon Gold 6326 CPU @ 2.9 GHz with 16 cores and 256 GB RAM.
We employed SCIP~9.0.0~\cite{SCIP} as the MIP solver.  
We also provide an alternative clustering approach based on the well-known \emph{k-means} algorithm~\cite{kmeans_macqueen}, using the scikit-learn~1.3.2 implementation~\cite{scikit-learn} with default parameters.
This implementation heuristically solves the problem of finding the best centroids with initializations based on kmeans++.

We demonstrate the effectiveness of our scenario reduction approach through numerical experiments on mixed-integer linear programs.  
We conduct experiments on a wide variety of problem instances and compare the performance of the solutions obtained using both the original and the reduced set of scenarios.  
Subsequently, the solutions are evaluated based on their objective values and the time required to solve the optimization problem.

To validate the performance of our approach, we used publicly available instances from the well-known benchmark library MIPLIB~\cite{gleixner2017miplib}.
We sorted the problems in the library by the number of variables and selected the smallest test instances that satisfied Assumption~\ref{assump:func_assumps} and could be solved within one hour of CPU time.
We generated different scenarios by perturbing the coefficients in the objective function uniformly at random by a multiplicative factor \(s_{inc}\).
The 17 instances selected are listed in Table~\ref{Table:overview_mip} in the appendix.
Consequently, the perturbed scenarios take random values in the interval 
$[(1-s_{inc})\cdot \bar{s},\ (1+s_{inc})\cdot \bar{s}]$,
where \(\bar{s}\in \R^n\) represents the original cost vector.
To construct ambiguity sets as described in Section~\ref{Sec:Confidence_Intervalls}, we compute multi-dimensional confidence intervals by drawing up to 10,000 realizations of the scenarios from an initially generated probability distribution \(p^*\).
Based on these realizations, we employ the analytic formula from~\cite{fitzpatrick1987quick} to construct the multi-dimensional confidence intervals:
\begin{equation}
	[l,u] = \left[\hat{p}_N - \frac{z_{\frac{\delta}{2}}}{2\sqrt{N}},\ \hat{p}_N + \frac{z_{\frac{\delta}{2}}}{2\sqrt{N}}\right].\label{conf_interval}
\end{equation}
Here, \(\hat{p}_N\) denotes the empirical probability distribution of \(p^*\) based on \(N\) samples, and \(z_{\frac{\delta}{2}}\) is the upper \((1-\delta/2)\)-percentile of the standard normal distribution.
In our experiments, we set the confidence level of the ambiguity sets to \(1-\delta = 0.9\).
Results for ellipsoidal ambiguity sets, constructed using the \(\ell_2\)-norm as described in Subsection~\ref{Sec:ellipsoidal}, exhibit similar characteristics and are therefore omitted here, but are displayed in the appendix.

\begin{figure}[H]
	\centering
	\includegraphics[width=0.9\linewidth]{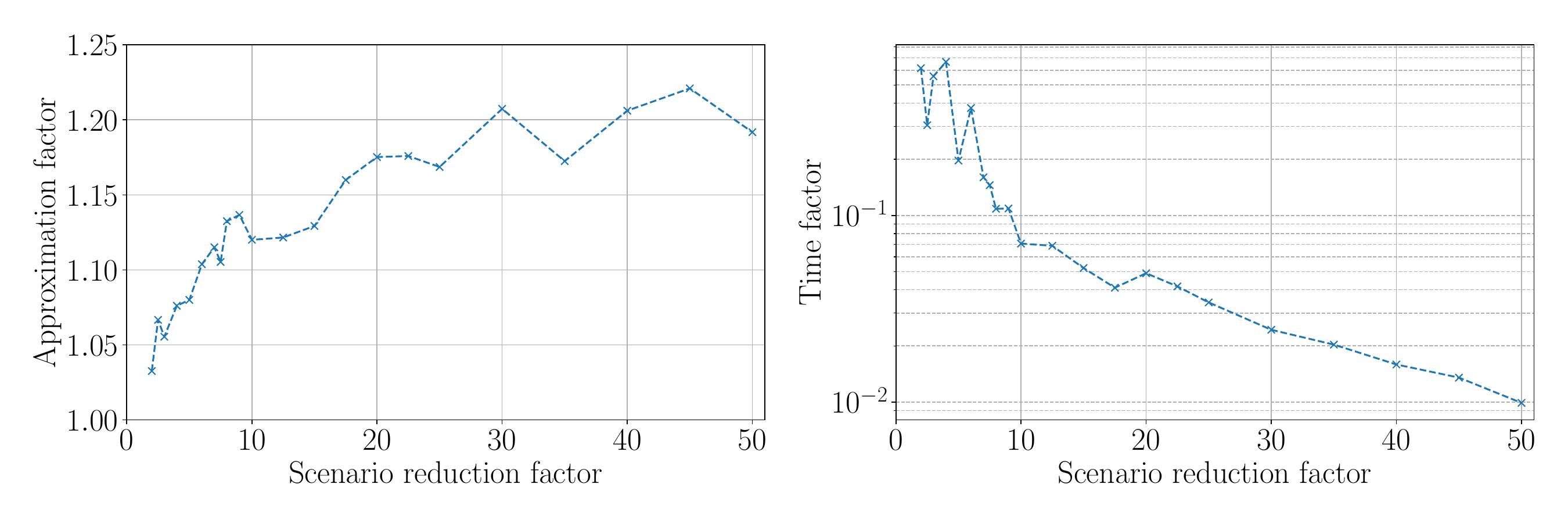}
	\caption{\label{fig:af_srf_0}The mean approximation factor (left) and time factor (right) in dependence of the scenario reduction factor for $N=0$ samples for \emph{opt} for the instance \emph{app2-2}.}
\end{figure}

The number of scenarios is varied in increments of 5, ranging from 5 to 50, while the number of clusters is set to 1, 2, or 5 (the case of reducing from 5 to 5 scenarios is omitted).
Thus, each instance is characterized by a unique combination of a MIPLIB instance, the number of scenarios, the number of reduced scenarios, and the value of \(s_{inc}\), which can be \(0.5\), \(0.75\), or \(0.9\).
Each such combination is solved for 10 different random initializations (seeds).

To evaluate the performance of the \emph{opt} and the \emph{k-means} approaches, we select DRO instances based on \emph{app2-2} from the MIPLIB library, using the probability simplex as the ambiguity set, as this instance exhibits the largest approximation errors, while other instances yield comparable or better results.
Figure~\ref{fig:af_srf_0} presents the mean approximation factor (AF) and the mean time factor (TF) as functions of the scenario reduction factor (SRF) for scenario reduction with \emph{opt}.
The mean is computed over all values of \(s_{inc}\) and averaged across all \(10\) random seeds.  
As expected, a higher scenario reduction factor, meaning a stronger reduction in the number of scenarios, leads to an increase in the approximation factor on average.  
For large scenario reduction factors such as \(45\), the approximation factor reaches approximately \(1.22\), indicating that reducing the number of scenarios results in a deterioration of up to \(22\%\) in the DRO objective function value compared to the original problem.
\emph{Opt} aims to find scenarios that minimize the approximation guarantee~\eqref{new_MINLPa}, which is an upper bound for the approximation factor. 
The solution time required to solve the DRO problem decreases significantly as the reduction factor increases due to the smaller problem size.
By balancing computational effort and accuracy, the solution time can be reduced by up to \(99\%\) for a scenario reduction factor of \(50\), while maintaining a relatively moderate approximation error.
Scenario reduction via \emph{$k$-means} leads to a similar reduction in solution time.

\begin{figure}[H]
	\centering
	\includegraphics[width=0.9\linewidth]{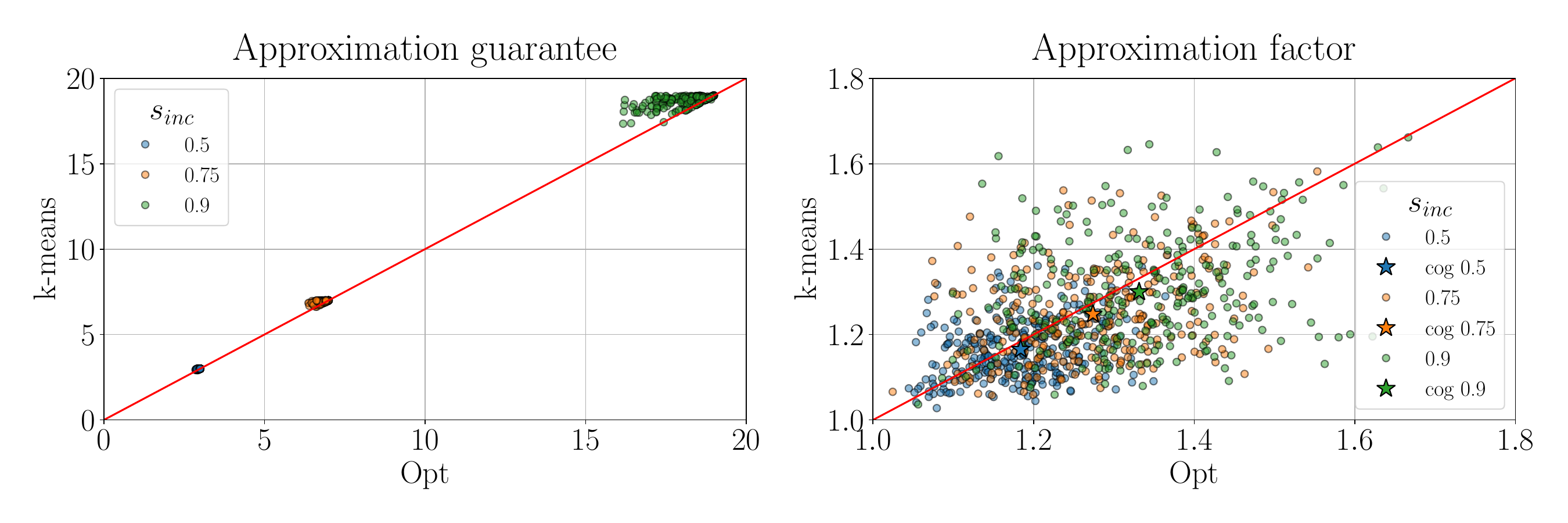}
	\caption{\label{fig:s_incr}The approximation guarantees (left) and the realized approximation factors (right) with the red line indicating the line through the origin for the instance \emph{app2-2}.
		The coloring of the dots indicates the different values of $s_{inc}$ used in the creation of the scenarios.
		The cogs indicate the average AFs for the three values of $s_{inc}$.}
\end{figure}

Next, we analyze the influence of \(s_{inc}\), which describes the magnitude of disturbance, on the approximation factor by plotting AF for \emph{opt} against AF for \emph{k-means}.  
The left scatter plot in Figure~\ref{fig:s_incr} reports the approximation guarantees given by the clustering problem~\eqref{Eq:Clustering_MINLP}, while the right plot displays the realized approximation errors measured by the approximation factor as defined in~\eqref{Eq:AF}.  
As before, we select the instance \emph{app2-2} with \(0\) samples.  
In both plots, the AFs for \emph{opt} and \emph{k-means} are plotted on the \(x\)- and \(y\)-axes, respectively.  
This representation allows for a direct comparison of both clustering methods: for any data point located below the red diagonal, \emph{k-means} yields a lower AF than \emph{opt}, while the opposite holds for points above the diagonal.  
Focusing on the left plot, we observe two key findings:  
First, all data points lie above the diagonal, which is expected, since \emph{opt} minimizes the worst-case AF.  
Second, \(s_{inc}\) has a strong impact on the worst-case AF.  
Although the influence of \(s_{inc}\) on the worst-case AF is substantial, its effect diminishes in practice, as seen in the right plot.  
Nevertheless, the impact remains noticeable for both clustering algorithms.  
There is no clear indication that either \emph{opt} or \emph{k-means} consistently achieves lower AFs.  
The gap between approximation guarantee and realized AFs is large, since the clustering process does not incorporate information about the ambiguity set or the feasible set $\mX$.  
Also, in many cases \emph{k-means} achieves similar results, as indicated by the cogs lying close to the diagonal in the right plot of Figure~\ref{fig:s_incr}.

\begin{figure}[H]
	\centering
	\includegraphics[width=0.9\linewidth]{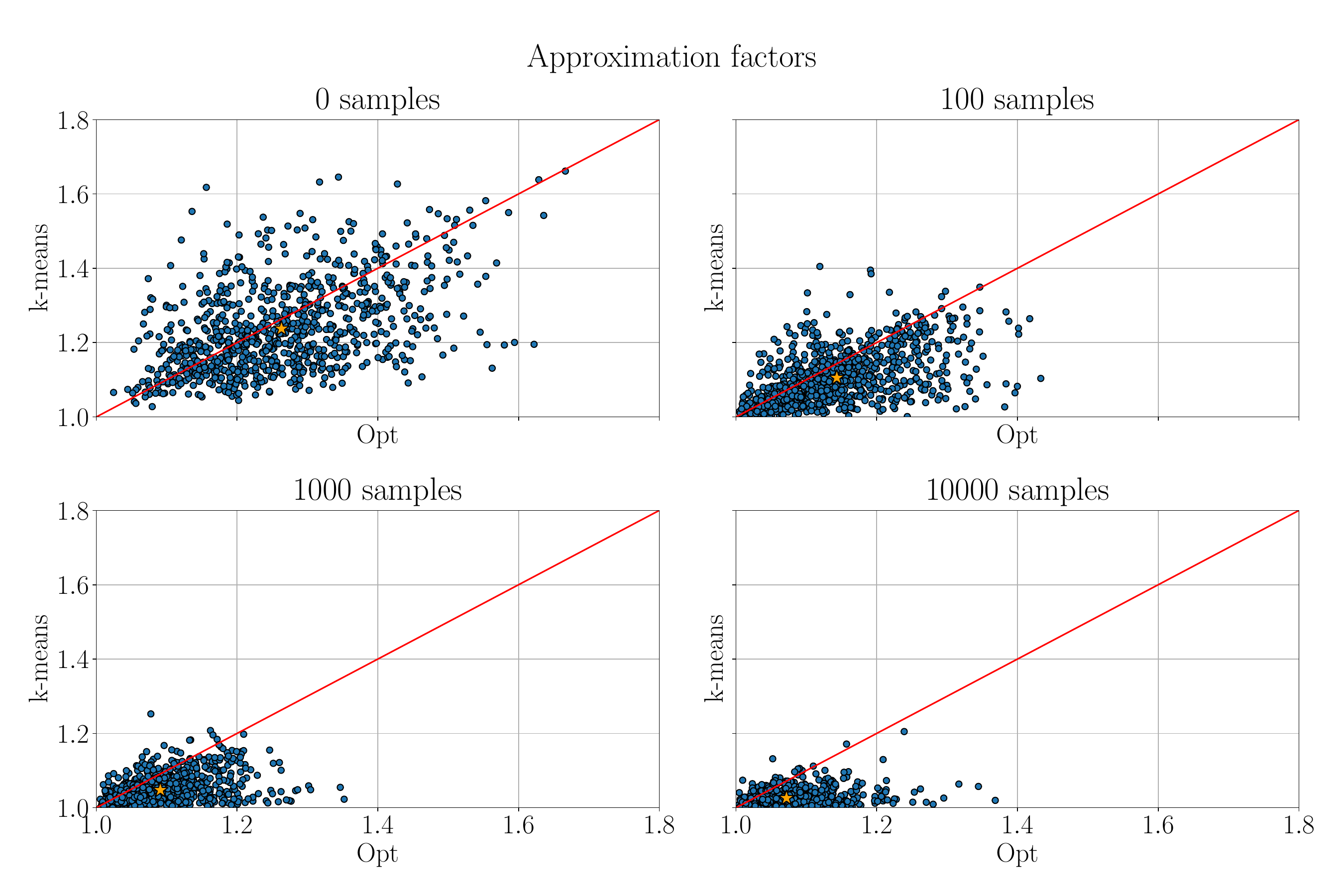}
	\caption{\label{fig:samples}Scatter plot of approximation factors for \emph{opt} and \emph{k-means} for the instance \emph{app2-2} with varying numbers of samples. The red line represents the diagonal through the origin, while the orange star marks the center of gravity of all displayed instances.}
\end{figure}

In Figure~\ref{fig:samples}, we plot the approximation factors for scenario reduction at different ambiguity set sizes, i.e., sets created using different numbers of sampled points (parameter \(N\) in~\eqref{conf_interval}).  
We vary the sample sizes in increments of \num{0}, \num{100}, \num{1000}, and \num{10000}.  
With a sample size of \(N=0\), we robustify against the entire probability simplex as the uncertainty set.  
Again, the red line passes through the origin and separates regions where \emph{opt} exhibits higher (above the line) or lower (below the line) approximation factors (AF).  
The center of gravity (cog) (orange) represents the average of all dotted points.  
As the number of samples increases, both the point cloud and the cog move towards the point \((1,1)\), indicating that the confidence intervals \eqref{conf_interval} shrink as more data is incorporated.  
This effect is observed for both clustering algorithms but is more pronounced for \emph{k-means} than for \emph{opt}.  
Since the cog remains below the line through the origin for all sample sizes, \emph{k-means} consistently achieves lower AFs than \emph{opt} on average.  

Finally, we report the average time needed to solve the MIP \eqref{Eq:Clustering_MIP} and the average time factor in Figure~\ref{fig:cluster_time} as a function of the number of initial scenarios for all 17 tested instances for \emph{opt} with one standard error.
Each curve represents the reduction to a different number of scenarios.
In the left plot, we see that the clustering time increases linearly with the number of scenarios.
The number of clusters, however, has a more significant impact on clustering time as indicated by comparing the line of $5$ clusters (green) to the line of one cluster (blue).
Initially, for 5 reduced scenarios, the solving time increases linearly.
Afterwards it levels off, since clustering times over 3600s were removed.
Comparatively, the median clustering time for \emph{opt} was much smaller at $8.4$~s.
For our experimental data, the TF decreases exponentially with the number of scenarios as already discussed in~\ref{fig:af_srf_0}, but gets larger with increasing number of reduced scenarios.
As can be expected, reductions to a small number of scenarios equate to the largest benefit in time reduction.

\begin{figure}[H]
	\centering
	\includegraphics[width=0.9\linewidth]{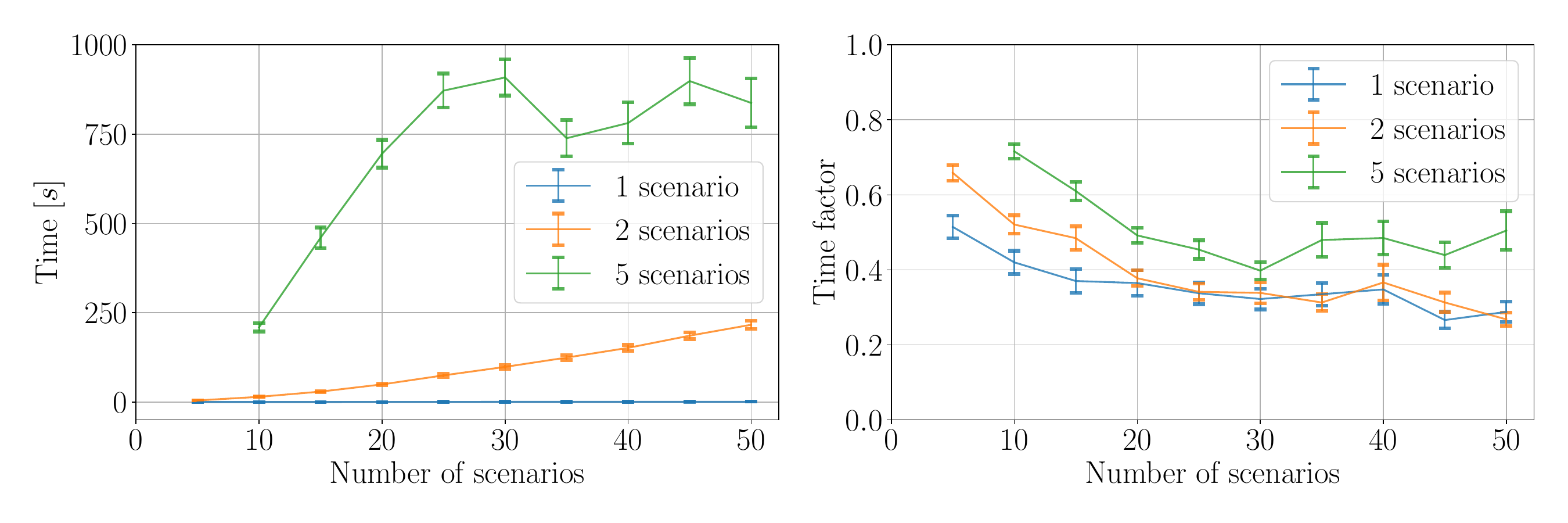}
	\caption{\label{fig:cluster_time}Clustering time in $s$ for solving the MIP \eqref{Eq:Clustering_MIP} (left) and the time factor (right) as a function of the number of scenarios for all tested instances.
		Each curve indicates a reduction to a different number of scenarios.}
\end{figure}

To summarize, the approximation factor increases with both the scenario reduction factor and \(s_{inc}\).
By design, \emph{opt} provides a smaller approximation guarantee than \emph{k-means}.
However, since the guarantee is independent of the ambiguity set and the feasible set~$\mX$, the observed approximation factors when solving a concrete MIPLIB instance are significantly lower than the theoretical bounds in Theorem~\ref{thm:abapprox}.  
Additionally, the AF decreases for both algorithms as the number of samples increases, since the confidence intervals in \eqref{conf_interval} shrink.  
Regarding the time factor, we observe an exponential decrease as a function of the scenario reduction factor and the total number of scenarios.  
Increasing the number of clusters affects both the time factor and the computational effort required for solving the clustering problem in \eqref{Eq:Clustering_MIP}.
As expected, the greatest time savings occur when reducing a large number of scenarios to a small number of clusters.  
Our numerical experiments show that the time required to solve an instance from MIPLIB can be reduced by a factor of up to 100, while the approximation factor remains below 1.35.  
We observe no significant differences in time factors when comparing the \emph{opt} and \emph{k-means} algorithms.  
Due to its low computational cost, \emph{k-means} is generally a well-suited approach for scenario reduction.  
However, if a theoretical approximation bound is required, scenario reduction via \emph{opt} is preferable.  

In our experimental evaluation, we restrict the comparison of heuristic methods to the well-known \emph{k-means} algorithm. 
The purpose of including \emph{k-means} is to provide a scalable surrogate for the exact mixed-integer clustering formulation, which serves as the theoretically grounded reference method with provable approximation guarantees. 
Thus, \emph{k-means} is not intended as a competitor among many heuristics, but as a representative scalable approximation of the exact approach when solving the MIP/MISDP becomes impractical. 
Including additional clustering heuristics would primarily amount to comparing different approximations of the same exact formulation, without adding further conceptual insight.

\subsection{Nonlinearly distributed scenarios}
To better understand the performance of the \emph{opt} and \emph{kmeans} methods, we extend our numerical calculations to objectives of the form
\begin{align}\label{eq:power}
	f(x, s) = x^\top s^\rho
\end{align}
for some exponent $\rho>1$ where we raise every component of $s$ to the power $\rho$.
These experiments are conducted on the \emph{MIPLIB} instance \emph{app2-2} as we see here the largest approximation errors in our experiments.
Scenarios $s$ are again generated uniformly at random, as described in Section~\ref{Section:MIPLIB}.

Figure~\ref{fig:bothplots} (left) presents the average approximation factors for opt and $k$-means as functions of the power~$\rho$ of the scenario~$s$ in the objective~\eqref{eq:power}.
The bars capture the standard error for each value of~$\rho$.

For $p \leq 1.5$, the approximation factor lies between 1.1 and 1.15 for both opt and $k$-means.
As $\rho$ increases, the approximation factors also rise, with \emph{opt} consistently yielding lower average values across all tested powers.
For larger values of~$\rho$, the advantage of \emph{opt} becomes more pronounced, as evidenced by the substantial gap in approximation factors at $\rho = 4$.

The superior performance of \emph{opt} arises from the nonlinearity of the objective:
raising~$s$ to a power~$\rho$ transforms the uniformly distributed scenarios into a nonlinear distribution in $\mathbb{R}^{|\mathcal{S}|}$, thereby reducing the effectiveness of the average-based scenario reduction performed by \emph{k-means}. 

Since \emph{opt} minimizes the worst-case approximation ratio on a logarithmic scale, it achieves greater robustness under these nonlinear transformations.
We therefore conjecture that \emph{opt} yields lower approximation factors in cases involving highly nonlinear scenario distributions. 

\begin{figure}[H]
	\centering
	\begin{subfigure}[t]{0.48\textwidth}
		\centering
		\includegraphics[width=\linewidth]{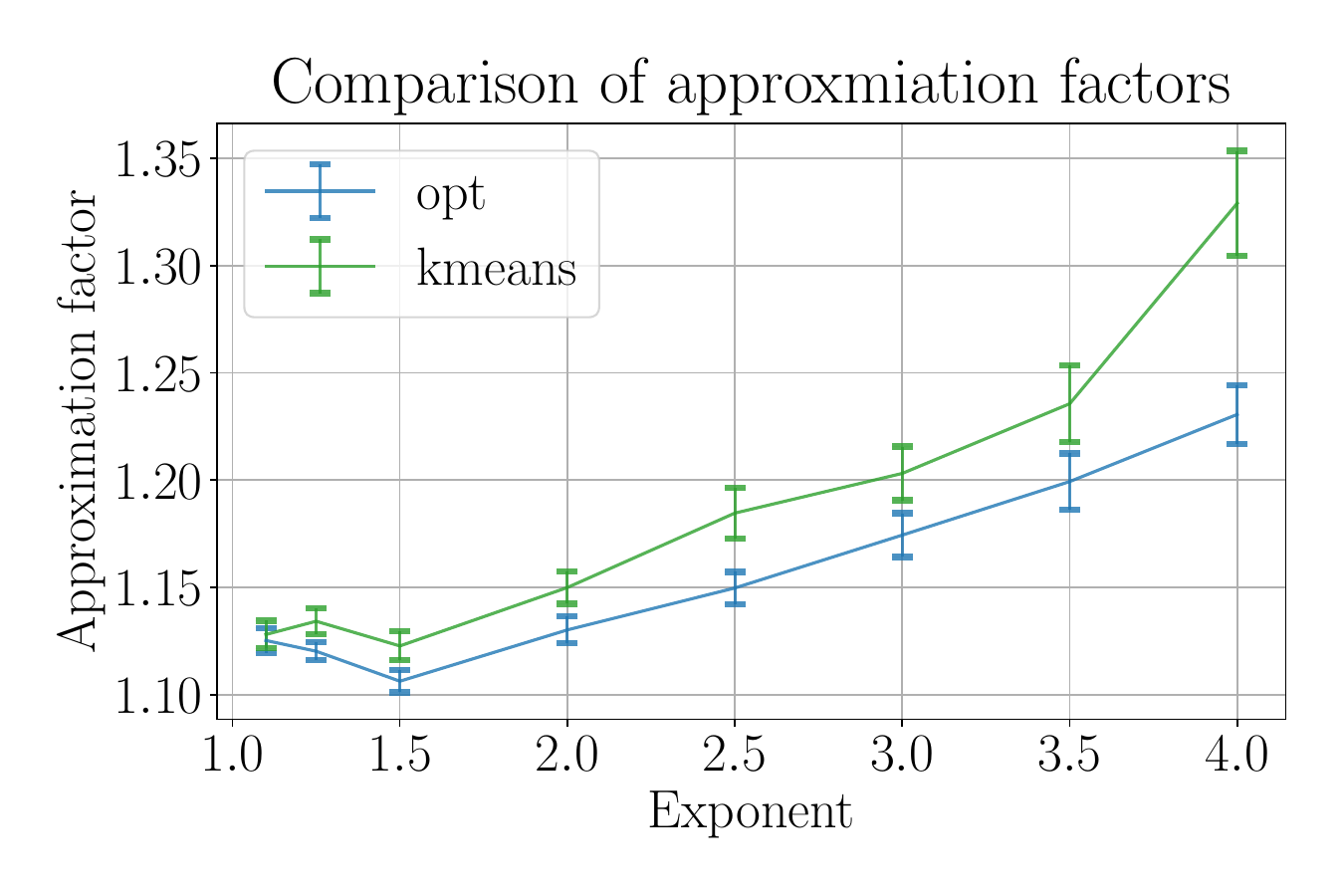}
	\end{subfigure}
	\hfill
	\begin{subfigure}[t]{0.48\textwidth}
		\centering
		\includegraphics[width=\linewidth]{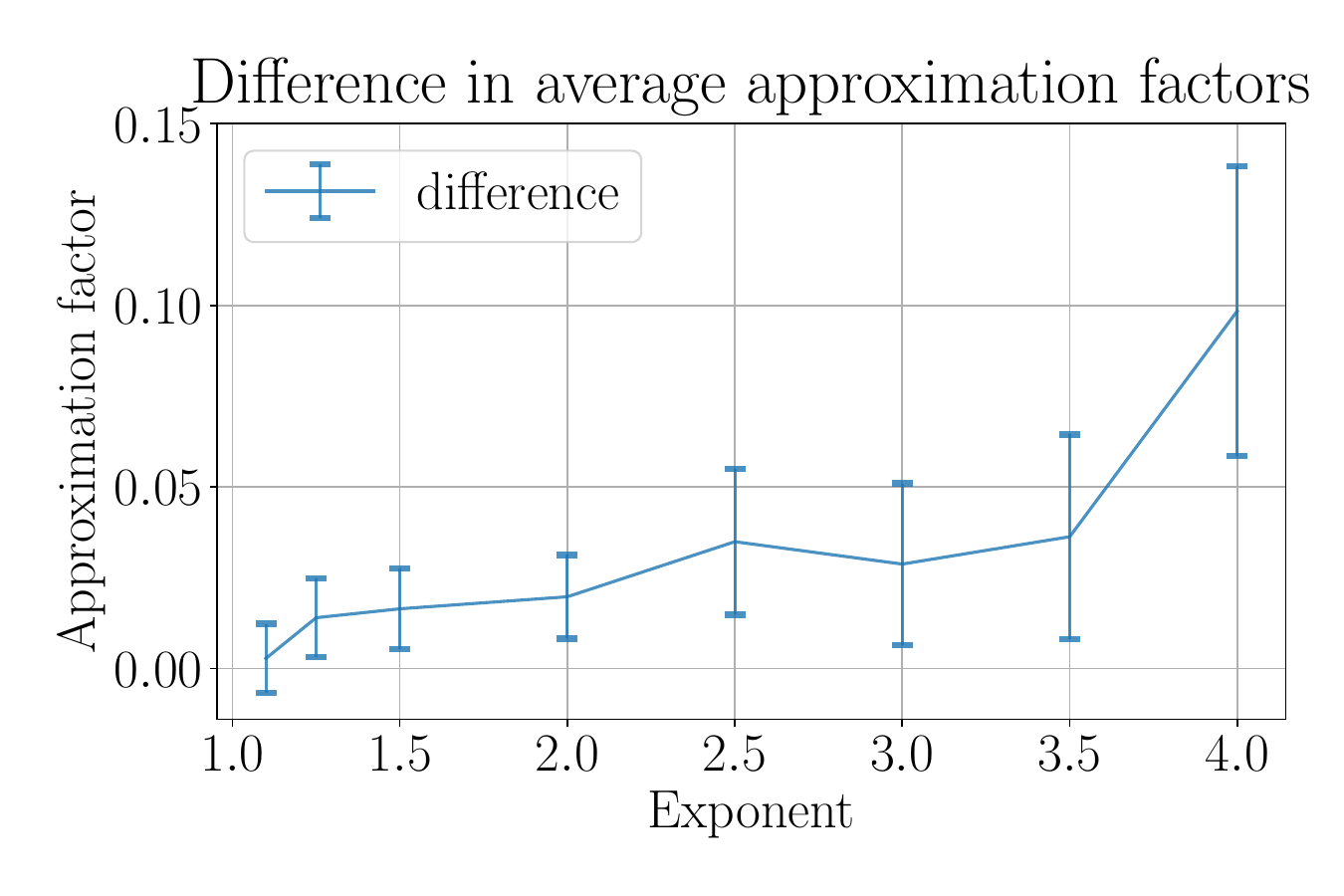}
	\end{subfigure}
	\caption{Approximation factors as a function of the exponent $\rho$ in \eqref{eq:power} for \textit{opt} and $k$-means (left), including standard errors for each exponent, and the corresponding average difference of the approximation factors (right), where positive values indicate that \textit{opt} achieves a lower approximation factor than $k$-means; two standard errors are shown for each exponent.}
	\label{fig:bothplots}
\end{figure}


In Figure~\ref{fig:bothplots} (right), we plot the average difference of approximation factors between \emph{opt} and $k$-means as functions of the power~$\rho$.
Positive values indicate lower approximation factor of \emph{opt} than $k$-means.
Furthermore, we provide two standard errors of the difference in average approximation factors for every $\rho$.
For every tested value $\rho\geq 1.25$, the lower bar is located above the zero line, indicating a significant difference between \emph{opt} and $k$-means.

For $\rho \leq 2$, the difference in average approximation factor lies between 0 and 0.025.
As $\rho$ increases, this difference also rises, indicating lower average approximation factors for \emph{opt} than $k$-means.
For larger values of~$\rho$, the advantage of \emph{opt} becomes more pronounced, as evidenced by the substantial increase at $\rho = 4$.

\subsection{Scenario Reduction for Portfolio Optimization}
\label{sec:port_opt} 
Portfolio optimization is one of the most prominent and practically relevant applications of distributionally robust optimization.
It serves as a widely studied benchmark for testing the performance of DRO methods, since it naturally involves uncertainty in expected returns and covariances.
In particular, the classical mean-variance formulation by Markowitz offers a structured yet realistic testbed to evaluate approximation quality and solution times under uncertainty.

We apply our scenario reduction approach to Markowitz's original formulation of portfolio optimization~\cite{markowitz}, which is given by
\begin{align*}
	\min_{\omega} \hspace{0.3cm} \omega^T \Sigma \omega \hspace{0.3cm}\text{s.t.} \hspace{0.3cm}\mu_0 \omega_0 + \mu^T \omega \geq R, \quad
	\sum_{i=0}^r \omega_i = 1, \quad \omega \geq 0,
\end{align*}
where \(\mu \in \mathbb{R}^r\) represents the expected returns of the given risky assets, and \(\mu_0\) is the risk-free return.
For our purposes, we use the average return of 1-year U.S. Treasuries over the selected time period.
Additionally, the covariance matrix of the risky assets is denoted by \(\Sigma \in \mathbb{R}^{r \times r}\).
The weights \(\omega_i\) represent the fraction of capital invested in asset \(i\), where \(i=0, 1, \ldots, r\).
The sum constraint \(\sum_{i=0}^r \omega_i = 1\) ensures full capital allocation.
The objective of this optimization problem is to minimize the portfolio variance \(\omega^T \Sigma \omega\) while guaranteeing an expected return of at least \(R \geq \mu_0\).

To construct distributionally robust test instances, we obtained stock market data using the \texttt{yfinance} module for \texttt{Python}.  
The selected period ranges from 2012 to 2023, with the final two years reserved for out-of-sample testing.  
To this end, and to mitigate survivorship bias, we retrieved the adjusted closing prices of all stocks that were part of the NASDAQ-100 at any point during this period.  
Scenarios were generated by randomly sampling a subset of stocks from this selection and computing the covariance matrix \(\Sigma\) of the stocks' time series for each year in the training dataset.  
Each resulting covariance matrix was treated as a possible scenario.  
For the construction of ambiguity sets using again Formula~\eqref{conf_interval}, we sampled from a uniform distribution over the scenarios.

Out of sample returns for the portfolio optimization problem are reported in Appendix.

\begin{figure}[H]
	\centering
	\includegraphics[width=0.9\textwidth]{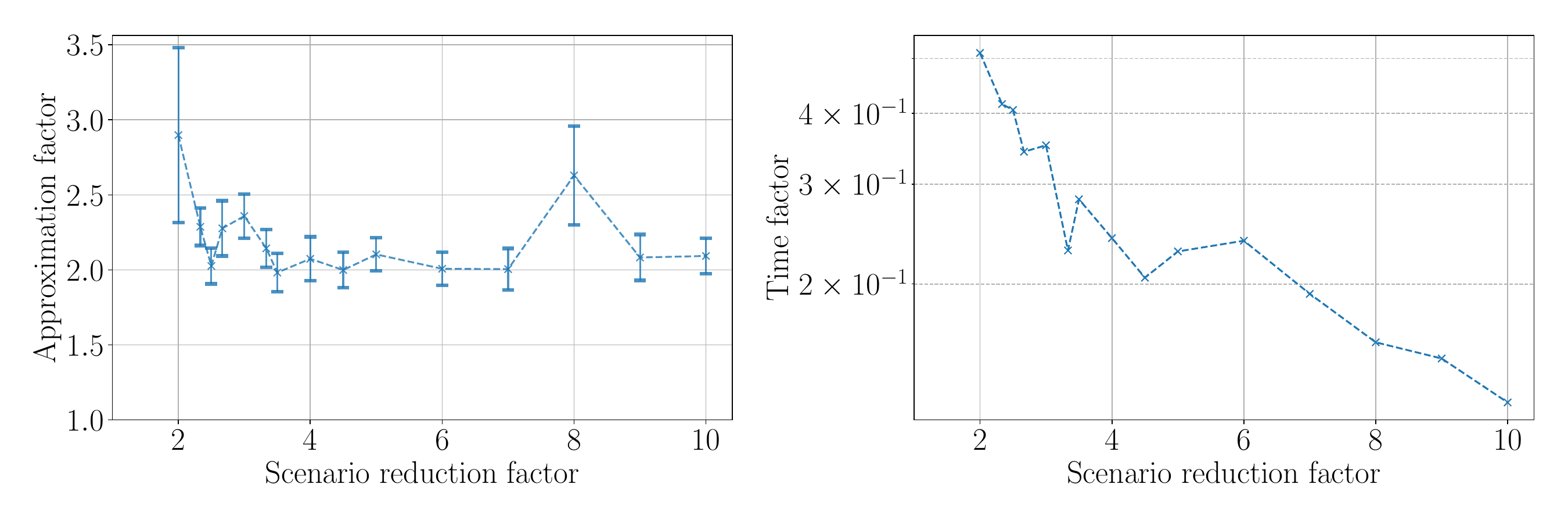}
	\caption{The mean time factor as a function of the scenario reduction factor (left) and the dimension of the covariance matrix (right) for \emph{opt}. 
		The time factor is calculated as the average across all instances for $N=0$ samples.\label{fig:time_factor}}
\end{figure}

Subsequently, the scenario matrices were clustered by solving \eqref{Eq:MISDP} using SCIP-SDP~4.3.0 and by applying the \emph{k-means} algorithm.  
Methodologically, the number of scenarios was set to 5, 6, 7, 8, 9, or 10 years, with the final year being 2021 and out-of-sample testing conducted in 2022 and 2023.  
The number of scenarios was then reduced to 1, 2, or 3, while the number of stocks, i.e., the dimensionality of \(\Sigma\), was randomly drawn from \(\{10, 11, \ldots, 20\}\) following a uniform distribution.  
The expected portfolio return was sampled uniformly from \([0.015, 0.25]\).  
This setup resulted in the creation of 800 portfolio optimization instances, with a maximum clustering time limit of 3600s.

All computations in this section were performed using a Python implementation on a machine equipped with an Intel® Core i7-8550U CPU @ 1.80GHz with 8 cores and 16 GB RAM.
We employed Gurobi~11.0.1~\cite{gurobi} as the solver for quadratic optimization problems and used the SCIP-SDP~4.3.0 plugin to solve mixed-integer semidefinite programs~\eqref{Eq:MISDP}.
As a benchmark, we compare our clustering approach against the well-known \emph{k-means} algorithm~\cite{kmeans_macqueen}, using the pyclustering~0.10.1.2 implementation~\cite{Novikov2019} and extending it to matrix clustering via the Frobenius norm.
We utilized the pyclustering package due to greater flexibility in editing the source code.

Analogous to the experiments in Section~\ref{Section:MIPLIB}, the results are analyzed with respect to the time savings achieved through scenario reduction (time factor (TF)) and the approximation error (approximation factor) of the in-sample DRO objective.
In the right subplot of Figure~\ref{fig:time_factor}, the time factor is shown as a function of the scenario reduction factor on a logarithmic scale for \emph{opt}.
	TF decreases linearly on the logarithmic axis thus the relative time required to solve~\eqref{Eq:DRO} decreases exponentially with increasing SRF.
	Furthermore, the time factor remains independent of both the number of samples drawn by the algorithm and the specific scenario reduction method employed.\\
	In the left subplot of Figure~\ref{fig:time_factor}, we plot the average approximation factor and the standard error as a function of the scenario reduction factor for \emph{opt}.
	The approximation factor remains relatively constant attaining values mostly between 2.0 and 2.5.
	This might be caused by the worst-case protection of the matrix clustering SDP~\eqref{Eq:MISDP}.
The median runtimes for scenario reduction were 8.4 s for \emph{opt} and 0.7 ms for \emph{k-means}, while the average runtimes were 62 s and 0.8 ms, respectively.  
However, the runtime required for \emph{opt} exhibits high variance, with the longest observed runtime reaching 2852 s, compared to only 2.6 ms for \emph{k-means}.
As expected, obtaining a cluster using \emph{k-means} is significantly faster than solving the MISDP~\eqref{Eq:MISDP}.  

In conclusion, our scenario reduction approaches perform well on the tested instances.  
They effectively reduce the complexity of quadratic DRO problems while maintaining a small approximation error.  
In practice, scenario reduction using \emph{k-means}  is a well-suited approach to achieving small approximation factors with limited computational effort.   
Moreover, as the number of samples increases, \emph{k-means} approaches an approximation factor of 1.  
Thus, \emph{k-means} serves as a high-quality heuristic for scenario reduction, particularly when minimizing the runtime of the global optimization problem is a priority.


\section{Conclusion} 
\label{sec:conclusion}
We introduced a scenario reduction method for distributionally robust optimization based on clustering the scenario set.  
This reduction provides a performance guarantee for the solution obtained when solving the problem with a reduced set of scenarios.  
For optimization problems with linear and quadratic objectives, we show that a minimized approximation guarantee and the corresponding reduced scenarios can be obtained by solving a linear and a semidefinite mixed-integer program, respectively.  
In the case of linear objectives, we derive theoretical bounds on the approximation guarantees and specify optimal representative scenarios.  
While our focus is on box and ellipsoidal ambiguity sets, our theoretical results are applicable to arbitrary sets.  
To evaluate the performance of the reduced scenarios, we applied our approach to instances from the MIPLIB library as well as to a convex quadratic portfolio optimization problem.
Our numerical results demonstrate a significant reduction in the runtime required to solve the DRO problem.  
Furthermore, we show that the \emph{k-means} algorithm serves as a high-quality heuristic for quickly solving the DRO problem while incurring only a small penalty in objective performance.
However, when the objective depends non-linearly on the scenario, our newly introduced approach yields significantly tighter approximation factors.
Currently, our approximation guarantees hold for arbitrary ambiguity sets, offering great flexibility but at the cost of accuray.
Additionally, we demonstrate the tightness of the bound introduced in our theoretical framework.

Future research will aim to incorporate problem-specific ambiguity set properties to improve the theoretical approximation guarantees.  
Additionally, the properties of the objective function and the feasible set $\mX$ could be more incorporated when reducing scenarios.  
Moreover, dimensionality reduction techniques could be applied to further decrease the runtime of the DRO problem.  


\section*{Acknowledgments}
We thank the German Research Foundation (DFG) for their support within Projects A05, B06, and B10 in the Sonderforschungsbereich/Transregio 154 Mathematical Modelling, Simulation and Optimization using the Example of Gas Networks with Project-ID 239904186.
This paper has been supported by the Deutsche Forschungsgemeinschaft (DFG, German Research Foundation) - Project-ID 416229255 - CRC 1411. 
Research reported in this paper was partially supported through the Research Campus Modal funded by the German Federal Ministry of Education and Research (fund numbers 05M14ZAM,05M20ZBM) and the Deutsche Forschungsgemeinschaft (DFG) through the DFG Cluster of Excellence MATH+.

\bibliographystyle{apalike}
\bibliography{references} 

\newpage
\section*{Appendix}
\renewcommand\thetheorem{\Alph{section}.\arabic{theorem}}
Here we present the proofs omitted from the main body of the paper as well as additional numerical experiments.

\subsection*{Theoretical Results}

The proof for Corollary~\ref{coro:diagonale}:

\begin{proof}
	%
	%
	We consider a point \(s'\) which lies on the line joining \(\underline{s}\) and \(\overline{s}\). 
	Then there exists a \(\lambda \geq 0\) such that 
	$s' = \lambda \underline{s} + (1 - \lambda)\overline{s}$.
	Given \(s'\), optimal approximation ratios
	\(\alpha\) and \(\beta\) for  \(s'\) are given by 
	\[
	\alpha = \max_{i=1,\dots,m}\left(\frac{\overline{s}_i}{s_i'}\right),\;\;\beta = \max_{i=1,\dots,m} \left(\frac{s_i'}{\underline{s}_i}\right).
	\]
	We will now prove \(\alpha \beta = \max_{i=1,\dots,m}{\overline{s}_i}/{\underline{s}_i}\).
	Let there be some \(j \in \{1, \ldots, m\}\) such that \(\alpha = {\overline{s}_{j}}/{s_{j}'}\). This means that \({\overline{s}_{j}}/{s_{j}'} \geq {\overline{s}_{i}}/{s_{i}'}\) for all \(i\neq j\). 
	Now we know that 
	\[
	s_i' = \lambda \underline{s}_i + (1 - \lambda) \overline{s}_i \quad \forall i=1,\ldots,m. 	
	\]
	Dividing on both sides by \(s_i'\) we obtain 
	\[
	1 = \lambda \frac{\underline{s}_i}{s_i'} + (1 - \lambda) \frac{\overline{s}_i}{s_i'}\quad \forall i=1,\ldots,m.
	\]
	Since all ratios (for all \(i\)) are equal to \(1\), then we can write
	\[
	\lambda \frac{\underline{s}_j}{s_j'} + (1 - \lambda) \frac{\overline{s}_j}{s_j'} = \lambda \frac{\underline{s}_i}{s_i'} + (1 - \lambda) \frac{\overline{s}_i}{s_i'}\quad \forall i \neq j.
	\]
	Rearranging the terms yields
	\[
	\lambda \left(\frac{\underline{s}_j}{s_j'} - \frac{\underline{s}_i}{s_i'}\right) =  (1 - \lambda)\left(\frac{\overline{s}_i}{s_i'} - \frac{\overline{s}_j}{s_j'}\right)\quad \forall i \neq j.
	\]
	Since we know that the index \(j\) maximizes the ratio
	\(\frac{\overline{s}_{j}}{s_{j}'}\), this means that \(\frac{\overline{s}_{j}}{s_{j}'} \geq \frac{\overline{s}_{i}}{s_{i}'} \;\;\forall i=1,\ldots,m\). 
	As such, the right-hand side of the above equation is less or equal to \(0\). 
	Thus, we have for the left-hand side that 
	\[
	\lambda \left(\frac{\underline{s}_j}{s_j'} - \frac{\underline{s}_i}{s_i'}\right) \leq 0 \quad \forall i \neq j.
	\]
	Flipping the numerator and the denominators, we obtain
	\[
	\frac{\underline{s}_j}{s_j'} \leq \frac{\underline{s}_i}{s_i'} \quad \forall i \neq j \;\implies\; \frac{s_j'}{\underline{s}_j} \geq \frac{s_i'}{\underline{s}_i} \quad \forall i \neq j.
	\]
	This means that 
	\[
	\beta = \max_{i=1,\dots,m} \left(\frac{s_i'}{\underline{s}_i}\right) = \frac{s_j'}{\underline{s}_j}.
	\]
	So we have that 
	\[\alpha \beta = \frac{\overline{s}_j}{s_j'} \cdot \frac{s_j'}{\underline{s}_j} = \frac{\overline{s}_j}{\underline{s}_j} = \max_{i=1,\dots,m} \left(\frac{\overline{s}_i}{\underline{s}_i}\right).\]
	This concludes the proof.
\end{proof}

\begin{lemma}
	\label{lemma:scaling_intervals}
	Scaling an axis by a constant $c>0$ doesn't change the axis' approximation bounds $\alpha$ and $\beta$.
\end{lemma}
\begin{proof}
	Consider a scenario set \(\mS \subseteq \R^n\). 
	Let 
	\[
	\underline{s}_i = \min_{s \in \mS} e_i^{\top}s \text{ and }	\overline{s}_i = \max_{s \in \mS} e_i^{\top}s.
	\]
	We also assume that $0<\underline{s}< \overline{s}$. 
	Let $\tilde{s} \in \mS$ be the optimal representative scenario for this set. 
	We know that \(\tilde{s}\) solves the following optimization problem 
	\begin{equation*}
	\begin{aligned}
		\min_{\alpha, \beta, \tilde{s}} &\; \alpha \beta \\
		\text{s.t.} &\; s \leq \alpha \tilde{s} \;\;\forall s \in \mathcal{S},\\
		&\; \tilde{s} \leq \beta s \;\;\forall s \in \mathcal{S}.
	\end{aligned}
	\end{equation*}
	Given \(\tilde{s}\) we can then write the optimal approximation ratios \(\alpha\) and \(\beta\) as 
	\begin{equation*}
	\alpha = \frac{\overline{s}}{\tilde{s}} \text{ and } \beta=\frac{\tilde{s}}{\underline{s}}.
	\end{equation*}
	Let \(c \in \R^n\) be a constant such that \(c > 0\). 
	Now, we scale the box by dividing component-wise by $c$. Then we can the scaled set \(\mS'\)
	\[
	\mS^\prime = \left\{\left(\frac{s_1}{c_1},\frac{s_2}{c_2},\dots,\frac{s_n}{c_n}\right) \mid s \in \mS\right\}.
	\]
	We first define a new cluster representative $\tilde{s}^\prime = \left(\frac{\tilde{s}_1}{c_1},\frac{\tilde{s}_2}{c_2},\dots,\frac{\tilde{s}_n}{c_n}\right)$ for the scaled set \(\mS'\).
	We can then obtain the new approximation bounds $\alpha^\prime$ and $\beta^\prime$ for \(\tilde{s}^\prime\).
	\begin{equation*}
	\alpha^\prime = \frac{\overline{s}/c}{\tilde{s}/c} = \frac{\overline{s}}{\tilde{s}}=\alpha \text{ and } \beta^\prime = \frac{\tilde{s}/c}{\underline{s}/c} = \frac{\tilde{s}}{\underline{s}} = \beta.
	\end{equation*}
	These are the same as the approximation ratios without scaling.
	
	We also observe that \(\alpha, \beta, \tilde{s}'\) are also the optimal solution to the problem 
	\begin{equation*}
		\begin{aligned}
			\min_{\alpha, \beta, \tilde{s}} &\; \alpha \beta \\
			\text{s.t.} &\; s \leq \alpha \tilde{s} \;\;\forall s \in \mathcal{S}',\\
			&\; \tilde{s} \leq \beta s \;\;\forall s \in \mathcal{S}',
		\end{aligned}
	\end{equation*}
	since the constraints are tight and the values of \(\alpha\) and \(\beta\) cannot be improved.
\end{proof}


	

\begin{remark}
	While the scaling in Lemma~\ref{lemma:scaling_intervals} was limited to one dimension, we can naturally extend it to multiple dimensions by using a different scaling constant \(c_i\) for dimension \(i\). 
	This allows us to rescale any set \(\mS\) such that its lower bound \(\underline{s}\) is always equal to \((1,1,\dots,1)\).
\end{remark}

\begin{theorem}
	\label{theorem:alpha_beta}
	Let $\alpha\beta$ be approximation bounds and define $\gamma_{\min}$ to be the minimal product of $\alpha\beta$. Then, the following statements hold for any cluster.
	\begin{enumerate}
	\item Any point $\tilde{s} \in \mS_j \subset \mathbb{R}^2$ is optimal w.r.t. $\alpha\beta$ if and only if $\tilde{s}$ is located between the two lines $y=x$ and $y=\frac{\overline{s}_y}{\overline{s}_x} \cdot x$ where $\overline{s} =(\overline{s}_x, \overline{s}_y)^T$ is the upper right corner and $(1 \;\; 1)^T$ is the lower left corner of the cluster.
	\item The following holds for two-dimensional problem instances:
	Suppose a cluster whose component-wise ratios of upper and lower bounds are equal. Then, only the points on the diagonal of the cluster are optimal w.r.t. $\alpha\beta$.
	\end{enumerate}
	\end{theorem}
	\vspace{.25cm}
\begin{proof}
	\emph{Part (1)}: Since we can scale any dimension by any factor $c>0$, we now assume the cluster box to be normalized, i.e., the lower left corner has a value of 1 in every component.\\
	W.l.o.g let $\overline{s}_y > \overline{s}_x$. Furthermore, let $\tilde{s} = (c, d)^T$ satisfy the above-mentioned conditions implying $d \geq c$. We obtain the following identity:
	\begin{equation*}
	\frac{d}{c} \leq \frac{\overline{s}_y}{\overline{s}_x} \iff \frac{\overline{s}_x}{c} \leq \frac{\overline{s}_y}{d}.
	\end{equation*}
	Calculating the approximation factors as well as their product yields
	\begin{align*}
	\alpha &= \max\left\{\frac{\overline{s}_x}{c}, \frac{\overline{s}_y}{d} \right\} = \frac{\overline{s}_y}{d}\\
	\beta &= \max\left\{c, d \right\} = d\\
	\alpha\beta &= \frac{\overline{s}_y}{d} \cdot d = \overline{s}_y
	\end{align*}
	$s_{u, y} = \max_{i}{\overline{s}^{(i)}} = \gamma_{\min}$ implies the proposition.\\
	We proceed by proving non-optimality of $\tilde{s}=(c \;\; d)^T$ w.r.t. $\alpha\beta$ for any $\tilde{s}$ outside of the aforementioned convex set.\\
	Again, w.l.o.g. let $\frac{d}{c} > \frac{\overline{s}_y}{\overline{s}_x} > 1$ and equivalently $\frac{\overline{s}_x}{c} > \frac{\overline{s}_y}{d}$. Calculating $\alpha$ and $\beta$ we obtain:
	\begin{align*}
	\alpha &= \max \left\{\frac{\overline{s}_x}{c}, \frac{\overline{s}_y}{d} \right\} = \frac{\overline{s}_x}{c}\\
	\beta &= \max \left\{ c, d \right\} = d\\
	\alpha\beta &= \frac{\overline{s}_x}{c} \cdot d > \frac{\overline{s}_y}{d} \cdot d = \overline{s}_y = \gamma_{\min}
	\end{align*}
	Finally, we keep the setting above, but select $\tilde{s}$ to be below $y=x$ implying $\frac{d}{c} < 1 < \frac{\overline{s}_y}{\overline{s}_x}$ and thus $\frac{\overline{s}_x}{c} < \frac{\overline{s}_y}{d}$. Calculating $\alpha$ and $\beta$ we obtain:
	\begin{align*}
	\alpha &= \max \left\{ \frac{\overline{s}_x}{c}, \frac{\overline{s}_y}{d} \right\} = \frac{\overline{s}_y}{d}\\
	\beta &= \max \left\{ c, d \right\} = c\\
	\alpha\beta &= \frac{\overline{s}_y}{d} \cdot c = \gamma_{\min} \cdot \frac{c}{d} > \gamma_{\min}
	\end{align*}
	\emph{Part (2)}: Since we can scale any axis by its lower bound, we can normalize the lower left corner to $(1 \;\; 1)^T$. Hence, the transformed cluster's upper right corner is given by $(d \;\; d)^T$ with $d>1$ and the two lines $y=x$ and $y=\frac{\overline{s}_y}{\overline{s}_x}x = \frac{d}{d} x = x$ from the previous statement (4) are the same. Therefore, the only optimal points w.r.t. $\alpha\beta$ lie on the diagonal of the cluster.
\end{proof}

\begin{lemma}\label{App:representatives}
	\begin{enumerate}
		\item A cluster representative $\tilde{s} \in \mathbb{R}^n_+$ is optimal w.r.t. $\gamma \coloneqq \alpha \beta$ if and only if $\argmax_i \tilde{s}^i = \argmax_i \frac{\overline{s}^i}{\tilde{s}^i}$.
		\item The set $C=\{\tilde{s} \in S \; \vert \; \gamma(\tilde{s}) = \gamma_{\min} \}$ of optimal cluster representatives is convex.
		\end{enumerate}
\end{lemma}
\begin{proof}
		\emph{Part (1)}: W.l.o.g., we normalize the lower left corner of the hyper rectangle to $(1, \ldots, 1)$. Furthermore, we can assume $\tilde{s}^{(1)} \geq \tilde{s}^{(2)} \geq \ldots \geq \tilde{s}^{(n)}$ by permutation of coordinates.\\
		Let's first assume $1 = \argmax_i \tilde{s}^{(i)} = \argmax_i \frac{\overline{s}^{(i)}}{\tilde{s}^{(i)}}$. Thus we obtain $\alpha=\frac{\overline{s}^{(1)}}{\tilde{s}^{(1)}}$ and $\beta = \tilde{s}^{(1)}$. Therefore $\alpha\beta = \overline{s}^{(1)}$ which is optimal as shown in 3).\\
		Now we let $j = \argmax_i \frac{\overline{s}^{(i)}}{\tilde{s}^{(i)}} \neq 1$ which implies $\frac{s_{u}^{(j)}}{\tilde{s}^{(j)}} > \frac{\overline{s}^{(1)}}{\tilde{s}^{(1)}}$. We still have $\beta = \tilde{s}^{(1)}$, but $\alpha = \frac{\overline{s}^{(j)}}{\tilde{s}^{(j)}}$. Hence
		\begin{align*}
		\alpha\beta = \frac{\overline{s}^{(j)}}{\tilde{s}^{(j)}} \tilde{s}^{(1)} > \frac{\overline{s}^{(1)}}{\tilde{s}^{(1)}} \tilde{s}^{(1)} = \overline{s}^{(1)} = \gamma_{\min}.
		\end{align*}
		\emph{Part (2)}: Let $\tilde{s}_1, \tilde{s}_2 \in C$. W.l.o.g. 6) implies 
		\begin{align*}
		\alpha(\tilde{s}_1) = \frac{\overline{s}^{(1)}}{\tilde{s}^{(1)}}, \;\;\;\; \beta(\tilde{s}_1) = \tilde{s}_1^{(1)}, \\
		\alpha(\tilde{s}_2) = \frac{\overline{s}^{(1)}}{\tilde{s}^{(1)}}, \;\;\;\; \beta(\tilde{s}_2) = \tilde{s}_2^{(1)}.
		\end{align*}
		For $\lambda \in [0, 1]$ let $t \coloneqq \lambda \tilde{s}_1 + (1-\lambda) \tilde{s}_2$ yielding
		\begin{align*}
		\alpha(t) &= \max_i \left\{ \frac{\overline{s}^{(i)}}{t^{(i)}} \right\} = \max_i \left\{ \frac{\overline{s}^{(i)}}{(\lambda \tilde{s}_1 + (1-\lambda) \tilde{s}_2)^{(i)}} \right\}\\
		\beta(t) &= \max_i \left\{ t^{(i)} \right\} = \lambda \tilde{s}_1^{(1)} + (1-\lambda) \tilde{s}_2^{(1)}
		\end{align*}
		We prove $\alpha(t) = \frac{\overline{s}^{(1)}}{(\lambda \tilde{s}_1 + (1-\lambda) \tilde{s}_2)^{(1)}}$ by contradiction. Assume there exists some $k \neq 1$ with
		\begin{align*}
		\frac{\overline{s}^{(k)}}{(\lambda \tilde{s}_1 + (1-\lambda) \tilde{s}_2)^{(k)}} &> \frac{\overline{s}^{(1)}}{(\lambda \tilde{s}_1 + (1-\lambda) \tilde{s}_2)^{(1)}}\\
		\iff \lambda \frac{\tilde{s}_1^{(k)}}{\overline{s}^{(k)}} + (1-\lambda) \frac{\tilde{s}_2^{(k)}}{\overline{s}^{(k)}} &< \lambda \frac{\tilde{s}_1^{(1)}}{\overline{s}^{(1)}} + (1-\lambda) \frac{\tilde{s}_2^{(1)}}{\overline{s}^{(1)}}.
		\end{align*}
		However since $\tilde{s}_1$, $\tilde{s}_2$ are optimal w.r.t. $\gamma$, the following holds for all $k \neq 1$:
		\begin{align*}
		\frac{\overline{s}^{(k)}}{\tilde{s}_1^{(k)}} < \frac{\overline{s}^{(1)}}{\tilde{s}_1^{(1)}} \text{  and  }
		\frac{\overline{s}^{(k)}}{\tilde{s}_2^{(k)}} < \frac{\overline{s}^{(1)}}{\tilde{s}_2^{(1)}}\\
		\iff \frac{\tilde{s}_1^{(k)}}{\overline{s}^{(k)}} > \frac{\tilde{s}_1^{(1)}}{\overline{s}^{(1)}} \text{  and  }
		\frac{\tilde{s}_2^{(k)}}{\overline{s}^{(k)}} > \frac{\tilde{s}_2^{(1)}}{\overline{s}^{(1)}}.
		\end{align*}
		Combining the previous two considerations, we obtain a contradiction:
		\begin{align*}
		\lambda \frac{\tilde{s}_1^{(1)}}{\overline{s}^{(1)}} + (1-\lambda) \frac{\tilde{s}_2^{(1)}}{\overline{s}^{(1)}} < \lambda \frac{\tilde{s}_1^{(k)}}{\overline{s}^{(k)}} + (1-\lambda) \frac{\tilde{s}_2^{(k)}}{\overline{s}^{(k)}} < \lambda \frac{\tilde{s}_1^{(1)}}{\overline{s}^{(1)}} + (1-\lambda) \frac{\tilde{s}_2^{(1)}}{\overline{s}^{(1)}}.
		\end{align*}
		Therefore $\alpha(t) = \frac{\overline{s}^{(1)}}{(\lambda \tilde{s}_1 + (1-\lambda) \tilde{s}_2)^{(1)}}$ must hold. And calculating
		\begin{align*}
		\alpha(t) \cdot \beta(t) = \frac{\overline{s}^{(1)}}{(\lambda \tilde{s}_1 + (1-\lambda) \tilde{s}_2)^{(1)}} \cdot \left(\lambda \tilde{s}_1^{(1)} + (1-\lambda) \tilde{s}_2^{(1)} \right) = \overline{s}^{(1)} = \gamma_{\min}
		\end{align*}
		implies $t \in C$. Thus $C$ is convex.
		%
		%
\end{proof}

\lemredellips*

\begin{proof}
	Consider the original ellipsoid contained in \(\mP\). We can express it as 
	\[
	\Omega = \{ p \mid (p - p^0)^{\top}\Sigma^{-1}(p - p^0) \leq r^2\}
	\]
	We know that the set of probability distributions over the reduced scenario set is a linear transformation of the original probability distribution with a full row rank matrix \(A\). 
	
	Using Lemma~\ref{lemma:ellipsoid_transformation}, we observe that the linear transformation of the ellipsoid \(\Omega\) with a full row rank matrix \(A\) is a new ellipsoid as follows
	\[
	\Omega = \{ \tilde{p} \mid (\tilde{p} - A p^0)^{\top}\tilde{\Sigma}^{-1}(\tilde{p} - A p^0) \leq r^2\}
	\]
	where \(\tilde{\Sigma} = A \Sigma A^{\top}\). 
	
	Since the matrix \(A\) aggregates the probability over the scenarios, we have that if the original set of probabilities satisfies \(\sum_{i=1}^{|\mS|} p_i=1\), then the aggregate set of probabilities \(\tilde{p} = A p\) will satisfy \(\sum_{i=1}^{|\tilde{\mS}|} \tilde{p}_i=1\). 

	This concludes the proof.
\end{proof}

\begin{lemma}
	\label{lemma:ellipsoid_transformation}
	\begin{enumerate}[label=(\roman*)]
	\item The transformation of a sphere centered at the origin by a full row rank matrix \(A\) is an ellipsoid centered at the origin with covariance matrix \(A A^{\top}\). 
	
	\item The map of an ellipsoid $P \in \mathbb{R}^n$ centered at \(x_0\) and with convariance matrix \(\Sigma\) under a full row-rank linear transformation $A \in \mathbb{R}^{m \times n}$ with $n \geq m$ is an ellipsoid in $R^m$ centered at \(A x_0\) and with covariance matrix \(A \Sigma A^{\top}\).
	\end{enumerate}
\end{lemma}
\begin{proof}
	\emph{Part (i):}
	Consider a sphere \(S\) given by 
	\[
	S = \{x \mid x^{\top}x \leq \epsilon^2\}	
	\]
	Given a full row rank matrix \(A\), we want to show that the transformation of \(S\) by the matrix \(A\) is an ellipsoid given by 
	\[
	E = \{y \mid y^{\top}(A A^{\top})^{-1}y \leq \epsilon^2\}.	
	\]
	To do so, we need to show that \(A S \subseteq E\) and \(E \subseteq A S\). 

	\emph{First Case}.
	Consider a point \(x \in S\). We want to show that \(A x \in E\).
	To do so, we need to show that. 
	\[
	x^{\top}A^{\top}(A A^{\top})^{-1}A x \leq \epsilon^2.	
	\]
	We know that due to the Cauchy-Schwarz inequality 
	\[
		x^{\top}A^{\top}(A A^{\top})^{-1}A x \leq \|A^{\top}(A A^{\top})^{-1}A\| \|x\|_2^2	
	\]
	where the first norm is the matrix norm.
	Due to Lemma~\ref{lemma:norm_bound}, we observe that \(\|A^{\top}(A A^{\top})^{-1}A\| \leq 1\). 
	Thus
	\[
		x^{\top}A^{\top}(A A^{\top})^{-1}A x \leq \|x\|_2^2	\leq \epsilon^2,
	\]
	which shows that \(A S \subseteq E\). 
	
	\emph{Second Case}.
	Now, we want to show that \(E \subseteq A S\). For this, we consider \(y \in E\) and show that there exists a corresponding \(x \in S\) such that \(y = Ax\).
	We define \(x\) as \(A^{+}y\). Since \(A\) has linearly independent rows then \(A^{+} = A^{\top}(A A^{\top})^{-1}\).
	We can then show that 
	\[
	A x = A A^{+} y = y, 	
	\]
	because \(A^{+}\) will be a right inverse of \(A\). 

	We now show that \(x \in S\). 
	To do so, we need to show that \(x^{\top}x \leq \epsilon^2\).
	We can then write 
	\begin{align*}
		x^{\top}x &= (A^{+}y)^{\top}A^{+} y \\
		&= y^{\top}(A^{\top}(A A^{\top})^{-1})^{\top}A^{\top}(A A^{\top})^{-1} y\\
		&= y^{\top}((A A^{\top})^{-1})^{\top}A A^{\top}(A A^{\top})^{-1} y\\
		&= y^{\top}((A A^{\top})^{-1})^{\top} y\\
		&= y^{\top}(A A^{\top})^{-1} y\\
		&\leq \epsilon^2.
	\end{align*}
	Here, the second equality is due to the definition of the pseudo-inverse, the third due to the transpose, and the fourth because \(AA^{\top}(AA^{\top})^{-1} = I\). 
	The last equality is because \(A A^{\top}\) and its inverse would be symmetric. 
	Finally, the last inequality exists because \(y \in E\). 
	This concludes the proof. 

	\emph{Part (ii):}
	Let $P=\{x \in \mathbb{R}^n \; \vert \, (x-x_0)^T \Sigma^{-1} (x-x_0) \leq \epsilon^2 \}$ be an ellipsoid in $\mathbb{R}^n$ centered at $x_0 \in \mathbb{R}^n$ with a positive definite matrix $\Sigma \in \mathbb{R}^{n \times n}$.
	Let $A \in \mathbb{R}^{m \times n}$ be a full row rank matrix.

	To prove the result, we show that the ellipsoid \(P\) can be expressed as the translation and a linear transformation of a sphere.
	
	We perform a variable transformation $x' \coloneqq x-x_0$:
	\begin{equation*}
		P' = \{x' \in \mathbb{R}^n \; \vert \; x'^T \Sigma_x^{-1}x' \leq \epsilon^2\}
	\end{equation*}
	Then we can observe that \(P = \{x \;\vert \; x_0 + x',\; x' \in P'\}\).
	Here, \(P'\) is an ellipsoid centered at the origin. 

	Since $\Sigma$ is positive definite, there exists an orthogonal decomposition
	\begin{equation*}
		\Sigma = S D S^T
	\end{equation*}
	with $S^T S = 1_n$ and $D \in \mathbb{R}^{n \times n}$ diagonal.
	With the matrix $D^{-1/2}$ containing the inverse square roots of the diagonal entries, we calculate:
	\begin{equation*}
		x'^T \Sigma^{-1} x' = x'^T S D^{-1} S^T x' = x'^T S D^{-1/2} D^{-1/2} S^T x' = (D^{-1/2} S^T x')^T (D^{-1/2} S^T x').
	\end{equation*}
	We now define
	\begin{equation*}
		B \coloneqq D^{-1/2} S^T \text{ and} \quad x'' \coloneqq Bx',
	\end{equation*}
	which allows us to write
	\begin{equation*}
		x'^T \Sigma_x x' = (Bx')^T(Bx') = x''^T x''
	\end{equation*}
	This  we can define the set \(P''\) as 
	\begin{equation*}
		P'' \; \coloneqq \; \{x'' \in \mathbb{R}^n \vert x'' = Bx', \; x \in P' \}.
	\end{equation*}
	Due to the positive definiteness of $\Sigma$, we can invert $B$:
	\begin{equation*}
		x' = B^{-1}x'' = S D^{1/2} x''.
	\end{equation*}
	This shows that 
	\[
	P' = B^{-1} P''.	
	\]

	This, combined with our earlier result, allows us to write 
	\[
	P = \{x \;\vert\; x = x_0 + B^{-1}x'',\; x'' \in P''\}.
	\] 
	More generally, we can state
	\[P = x_0 + P' = x_0 + B^{-1}P''\].

	Now, we consider the linear transformation of \(P\).
	We then have 
	\[
	A P = \{y \;\vert\; y = A x_0 + A B^{-1}x'',\; x'' \in P''\}.	
	\]

	We define \(A' = A B^{-1}\). Since \(A\) is full row rank and \(B^{-1}\) is invertible, \(A B^{-1}\) has full row rank.
	Then, we can write
	\[
	A P = \{y \;\vert\; y = A x_0 + A' x'',\; x'' \in P''\}.	
	\]

	Due to Lemma~\ref{lemma:ellipsoid_transformation} and since \(P''\) is a sphere, the transformation of the sphere \(P''\) by a full row rank matrix \(A'\) is an ellipsoid with covariance matrix \(A' (A')^{\top}\).
	We have 
	\begin{align*}
		A' (A')^{\top} &= A B^{-1} (B^{-1})^{\top}A^{\top}\\
		&= A S D^{1/2} (S D^{1/2})^{\top}A^{\top}\\
		&= A S D^{1/2}  (D^{1/2})^{\top}S^{\top} A^{\top}\\
		&= A S D S^{\top} A^{\top}\\
		&= A \Sigma A^{\top}
	\end{align*}.

	Thus the linear transformation of \(P\) by a full row rank matrix \(A\) is an ellipsoid with center \(A x_0\) and covariance matrix \(A \Sigma A^{\top}\).
\end{proof}

\begin{lemma}
	\label{lemma:norm_bound}
	For any full row rank matrix \(A\), we have that \(\|A^{\top}(A A^{\top})^{-1}A\| \leq 1\) where \(\|A\| = \sup_{\|x\|_2 = 1} \|A x\|_2\). 
\end{lemma}
\begin{proof}
	Using the singular value decomposition of \(A\), we can write \(A = U \Sigma V^{\top}\). 
	Here \(U \in \mathbb{R}^{m \times m}\) and \(V \in \mathbb{R}^{n \times n}\) are orthogonal matrices and \(\Sigma \in \mathbb{R}^{m \times n}\) is a rectangular diagonal matrix of the form \([\Delta, 0]\) where \(\Delta \in \mathbb{R}^{m \times m}\) is square matrix with non-zero diagonals. 
	
	We start by considering \(A A^{\top})^{-1}\). 
	Then, we can write 
	\begin{align*}
		(A A^{\top})^{-1} &= (U \Sigma V^{\top}V \Sigma^{\top} U^{\top})^{-1}\\
		&= (U \Sigma \Sigma^{\top} U^{\top})^{-1}\\
		&= (U \Delta^2 U^{\top})^{-1}\\
		&= (U^{\top})^{-1} \Delta^{-2} U^{-1}\\
		&= U \Delta^{-2} U^{T}.
	\end{align*}
	Here, the second equality is because \(V\) is orthogonal, and the third is due to the structure of \(\Sigma\).
	We can then write 
	\begin{align*}
		A^{\top}(A A^{\top})^{-1}A &= V \Sigma^{\top} U^{\top} U \Delta^{-2} U^{T} U \Sigma V^{\top}\\
		&= V \Sigma^{\top} \Delta^{-2} \Sigma V^{\top}\\
		&= V \begin{bmatrix} I_m & 0 \\ 0 & 0 \end{bmatrix} V^{\top}.
	\end{align*}
	Observe that each of the three matrices has a maximum eigenvalue of \(1\) as \(V\) is orthogonal and because the middle matrix is a diagonal matrix with elements \(0\) or \(1\). 
	Thus, due to the Cauchy-Schwarz inequality
	\[
	\|A^{\top}(A A^{\top})^{-1}A\| \leq \|V\|\|\begin{bmatrix} I_m & 0 \\ 0 & 0 \end{bmatrix}\|\|V^{\top}\| \leq 1.
	\]
	This concludes the proof.
\end{proof}

\subsection*{Numerical Experiments}

\begin{table}[hbt]
	\centering
	\begin{tabular}{lrrrrr}
		\toprule
		Name &  \multicolumn{4}{c}{Variables} & Constraints \\
		\midrule
		& All & Bin. & Int. & Cont. & \\ \midrule
		10teams & 2025 & 1800 & 0 & 225 & 230 \\
		app2-2& 1226 & 1226 & 0 & 0 & 335 \\
		exp-1-500-5-5 & 990 & 250 & 0 & 740 & 550 \\
		flugpl& 18 & 0& 11& 7 & 18 \\
		g200x740 & 1480 & 740 & 0 & 740 & 940 \\
		gen-ip054 & 30 & 0 & 30 & 0 & 27 \\
		gr4x6 & 48 & 24 & 0 & 24 & 34 \\
		iis-glass-cov & 214 & 214 & 0 & 0 & 5375 \\
		k16x240b & 480 & 240 & 0 & 240 & 256 \\
		mtest4ma & 1950 & 975 & 0 & 975 & 1174 \\
		neos-860300 & 1385 & 1384 & 0 & 1 & 850 \\
		neos5 & 63 & 53 & 0 & 10 & 63 \\
		p0201 & 201 & 201 & 0 & 0 & 133 \\
		r50x360 & 720 & 360 & 0 & 360 & 410 \\
		ran12x21 & 504 & 252 & 0 & 252 & 285 \\
		ran13x13 & 338 & 169 & 0 & 169 & 195 \\
		seymour1 & 1372 & 451 & 0 & 921 & 4944 \\
		\bottomrule
	\end{tabular}
	\caption{\label{Table:overview_mip}Overview of MIP instances. Each problem was solved for 10 seeds \(\times\) 3 $s_{inc}$ \(\times\) 29 scenario reductions = 870 times for both ambiguity sets, i.e. box and $\ell_2$.}
\end{table}

\begin{figure}[htb]
	\centering
	\includegraphics[width=0.7\linewidth]{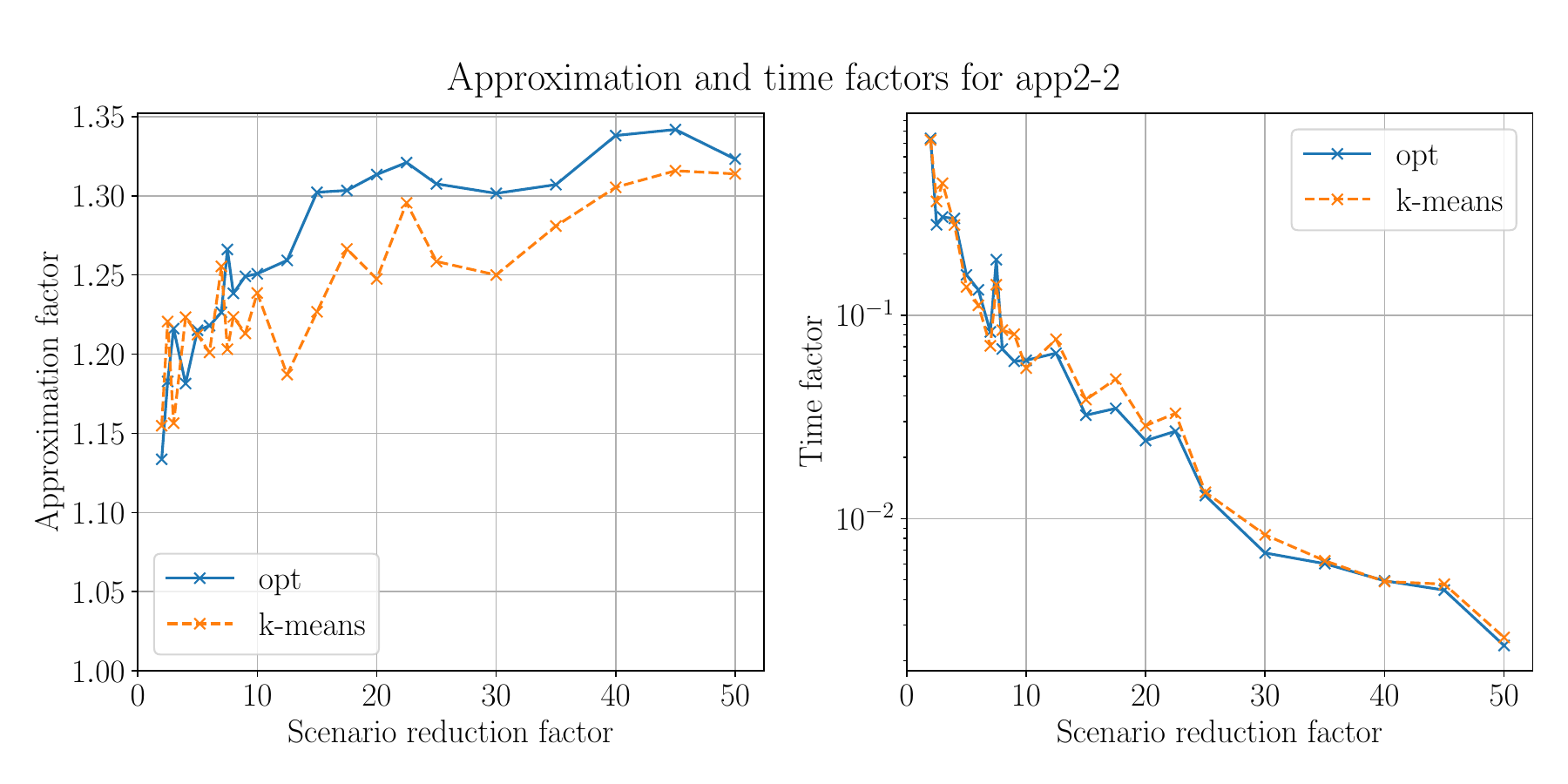}
	\caption{The mean approximation factor (left) and time factor (right) in dependence of the scenario reduction factor for $N=0$ samples for the two different clustering algorithms for the instance \emph{app2-2} and for $\ell_2$-ambiguity sets.}
\end{figure}

\begin{figure}[htb]
	\centering
	\includegraphics[width=0.7\linewidth]{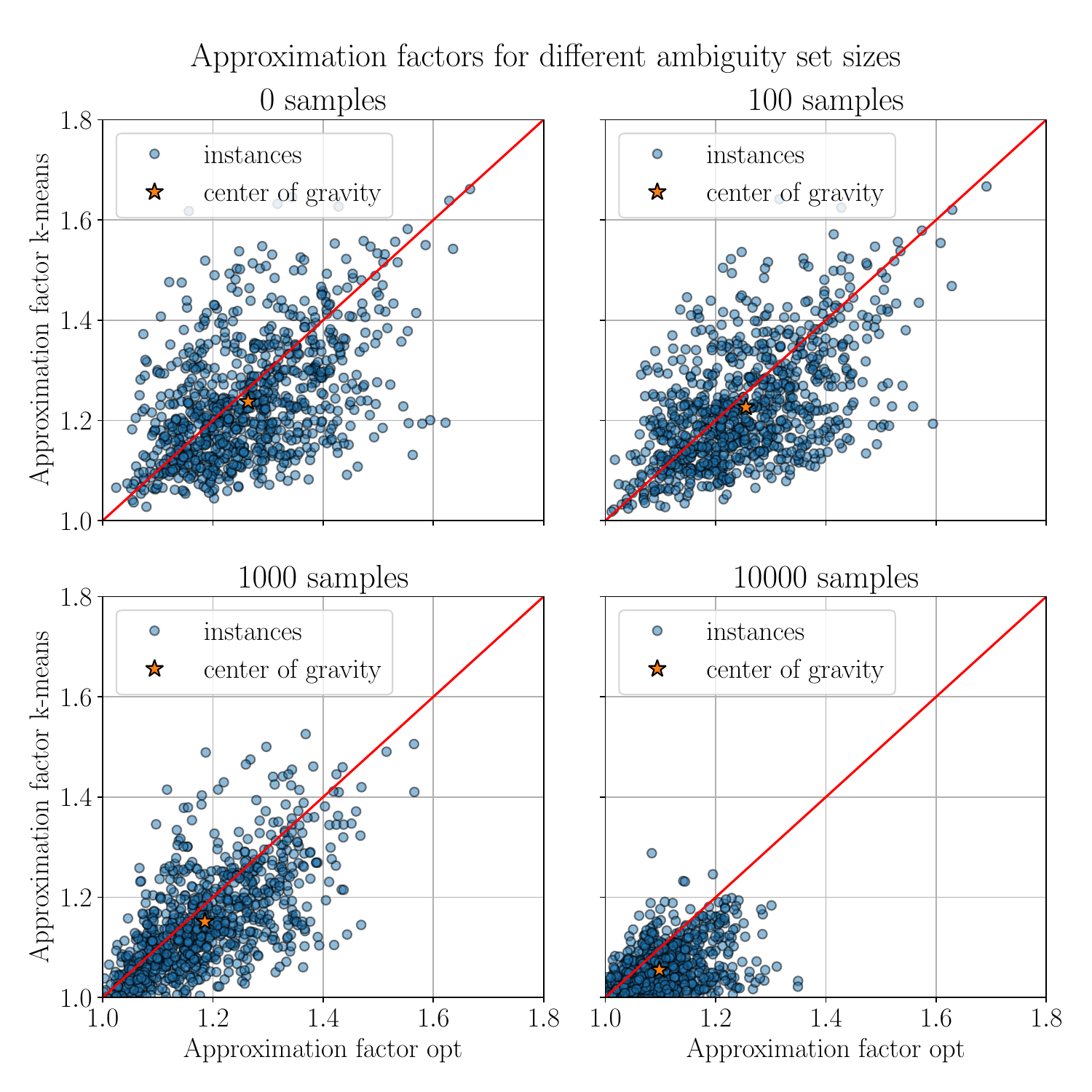}
	\caption{Scatter plot of approximation factors for \emph{opt} and \emph{k-means} for the instance \emph{app2-2} with varying numbers of samples. The red line represents the diagonal through the origin, while the orange star marks the center of gravity of all displayed instances and for $\ell_2$-ambiguity sets.}
\end{figure}
\begin{figure}[htb]
	\centering
	\includegraphics[width=0.7\textwidth]{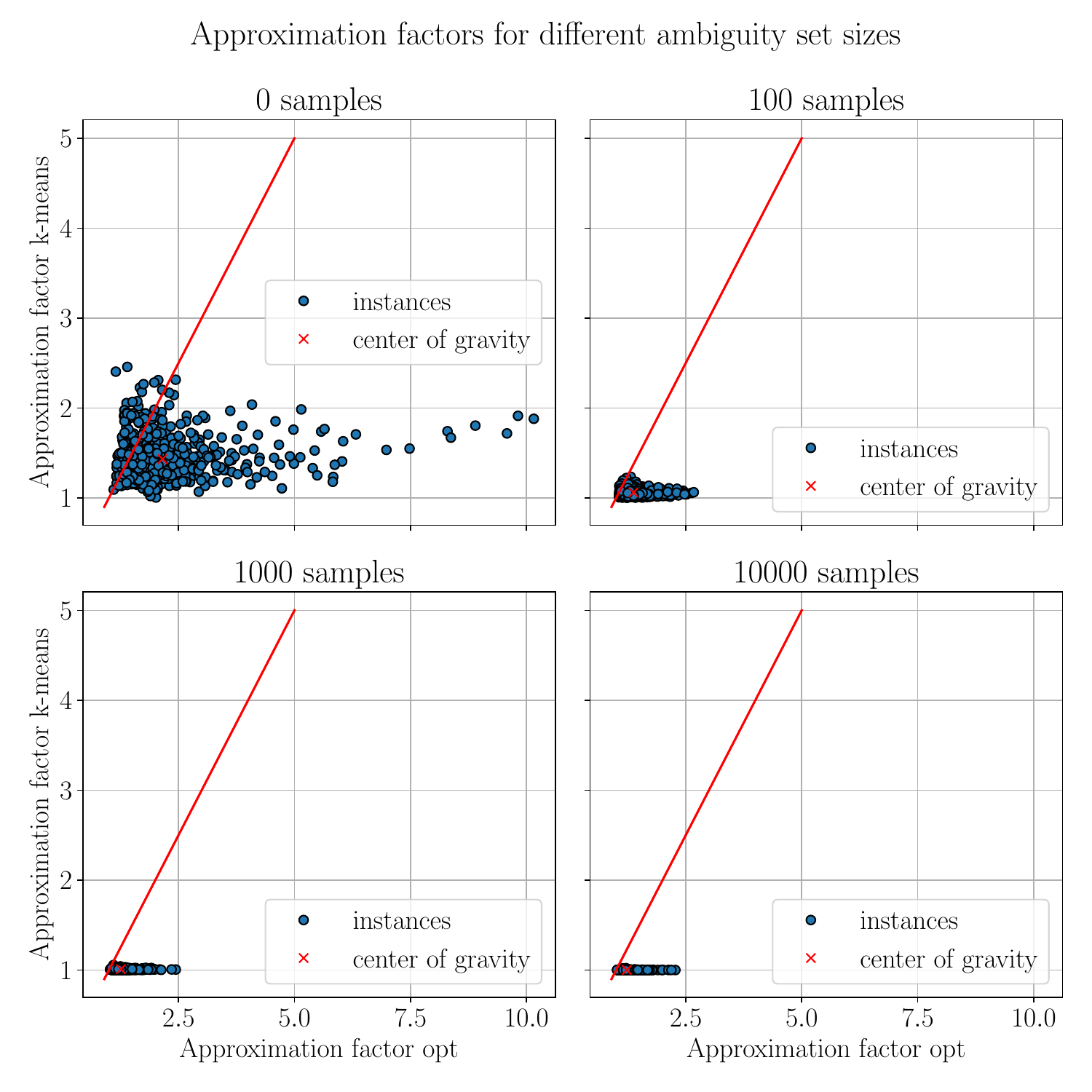}
	\caption{\label{fig:AF_portfolio}Comparison of the realized approximation factors for \emph{opt} and \emph{k-means} for all 800 portfolio instances for different number of samples.
		The red line represents the diagonal where $AF_{opt}=AF_{k-means}$ and the red cross marks the center of gravity of all data points.
	}
\end{figure}

\begin{figure}
	\centering
	\includegraphics[width=0.7\linewidth]{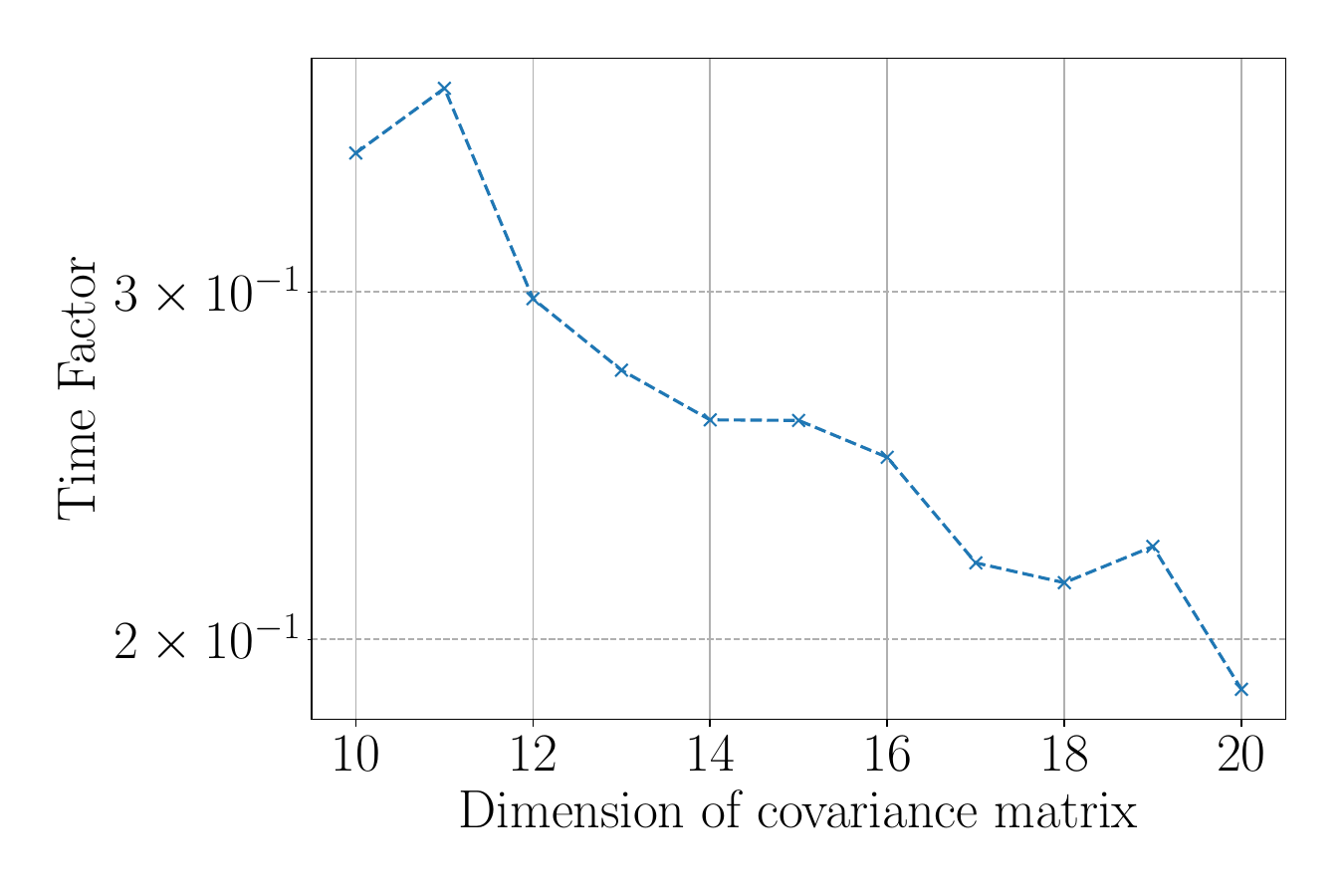}
	\caption{The mean time factor as a function of the dimension of the covariance matrix for all tested instances of portfolio optimization.
	}
\end{figure}

\begin{figure}[htb]
	\centering
	\includegraphics[width=0.7\linewidth]{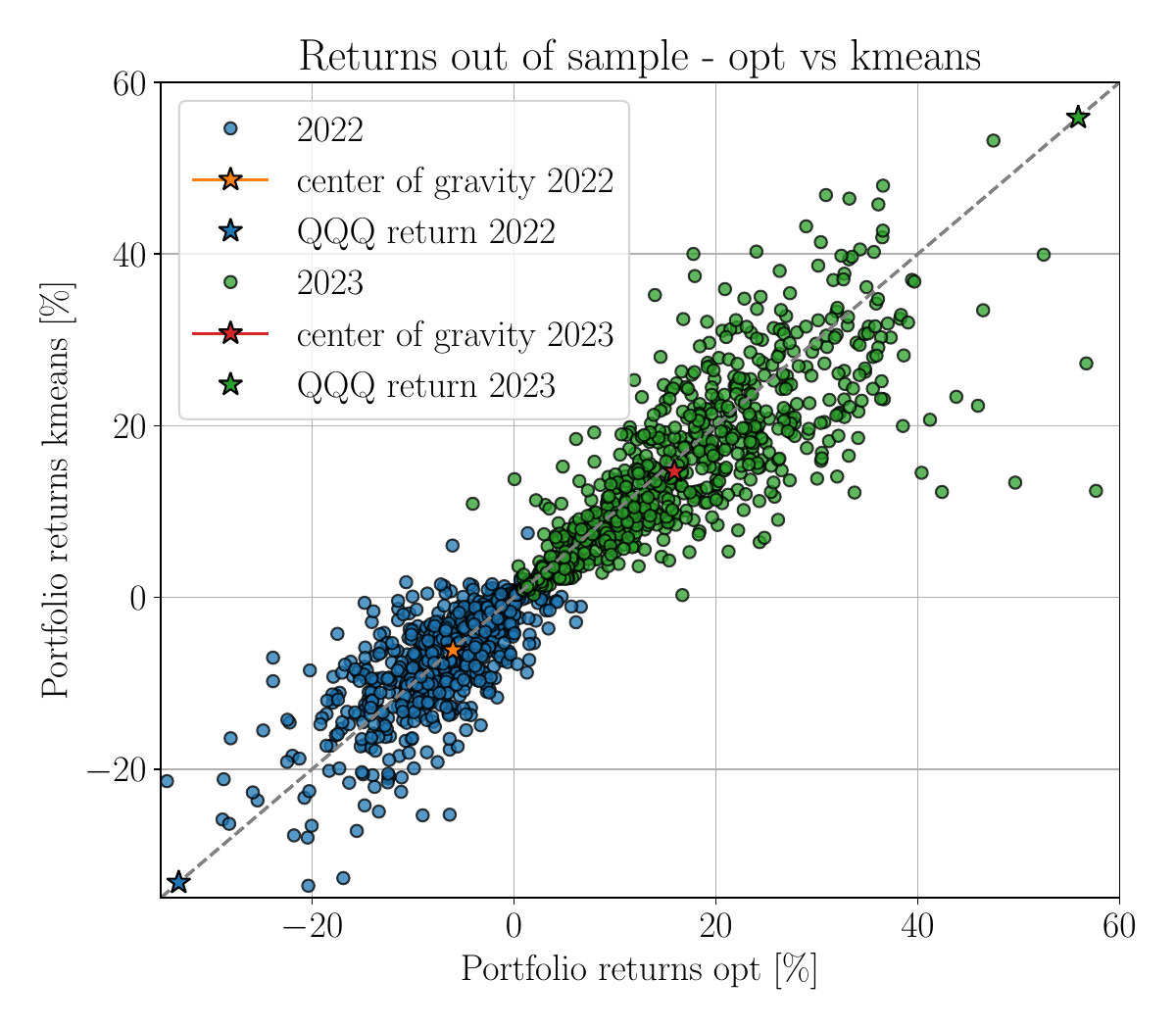}
\caption{Out of sample returns of the portfolio for opt and k-means for the years 2022 (blue) and 2023 (green).
	The blue and green stars indicate the returns of the NASDAQ-100 in the respective years.
	The orange and red stars indicate the centers of gravity of portfolio returns for the years 2022 and 2023, respectively.}
\end{figure}



\end{document}